\documentclass[11pt]{article}
\usepackage[left=1in,right=1in,top=1in,bottom=1in]{geometry}
\usepackage{times}
\usepackage{expl3}
\usepackage{cite}
\usepackage[table]{xcolor}
\usepackage{multirow}
\usepackage{stackengine} 
\usepackage{hhline}
\usepackage{lipsum}
\usepackage{titlesec}
\usepackage{wrapfig}
\usepackage{epsfig}
\usepackage{graphicx}
\usepackage{amsmath}
\usepackage{amssymb}
\usepackage{epstopdf}
\usepackage{boldline}
\usepackage{calligra}
\usepackage{url}
\usepackage{blindtext}

\newcommand{\define}{\stackrel{\mbox{\tiny def}}{=}}

\newtheorem{definition}{Definition}

\newtheorem{corollary}{Corollary}
\newtheorem{lemma}{Lemma}
\newtheorem{proof}{Proof}

\usepackage{mathtools}
\usepackage{epstopdf}
\usepackage{balance}
\usepackage{thmtools}
\usepackage{thm-restate}
\usepackage{hyperref}
\usepackage{cleveref}

\usepackage[ruled,vlined]{algorithm2e}
\include{pythonlisting}
\newcommand{\ostar}{\mathbin{\mathpalette\make@circled\star}}
\newcommand\oast{\stackMath\mathbin{\stackinset{c}{0ex}{c}{0ex}{\ast}{\bigcirc}}}

\makeatletter
\newcommand{\removelatexerror}{\let\@latex@error\@gobble}
\makeatother
\makeatletter
\newcommand*{\rom}[1]{\expandafter\@slowromancap\romannumeral #1@}
\makeatother

\ExplSyntaxOn
\newcommand\latinabbrev[1]{
  \peek_meaning:NTF . {% Same as \@ifnextchar
    #1\@}%
  { \peek_catcode:NTF a {% Check whether next char has same catcode as \'a, i.e., is a letter
      #1.\@ }%
    {#1.\@}}}
\ExplSyntaxOff

\titleclass{\subsubsubsection}{straight}[\subsubsection]

\begin{document}
\vspace{1cm}
\title{T-product Tensors—Part I: Inequalities}\vspace{1.8cm}
\author{Shih~Yu~Chang\thanks{Shih Yu Chang is with the Department of Applied Data Science,		San Jose State University, San Jose, CA, U. S. A. E-mail: 
			shihyu.chang@sjsu.edu} \quad and \quad Yimin Wei
\thanks{ Corresponding author (Y. Wei). E-mail:  ymwei@fudan.edu.cn, yimin.wei@gmail.com. Yimin Wei is with the School of Mathematical Sciences and Shanghai Key Laboratory of Contemporary Applied Mathematics, Fudan University, Shanghai, 200433, PR China. }}                    
\maketitle

\date{}

\begin{abstract}
The T-product operation between two three-order tensors was invented around 2011 and it arises from many applications, such as signal processing, image feature extraction, machine learning, computer vision, and the multi-view clustering problem. Although there are many pioneer works about T-product tensors, there are no works dedicated to inequalities associated with T-product tensors. In this work, we first attempt to build inequalities at the following aspects: (1) trace function nondecreasing/convexity; (2) Golden-Thompson inequality for T-product tensors; (3) Jensen’s T-product inequality; (4) Klein’s T-product inequality. All these inequalities are related to generalize celebrated Lieb’s concavity theorem from matrices to T-product tensors. This new version of Lieb's concavity theorem under T-product tensor will be used to determine the tail bound for the maximum eigenvalue induced by independent sums of random Hermitian T-product, which is the key tool to derive various new tail bounds for random T-product tensors. Besides, Qi et. al~\cite{qi2021tquadratic} introduces a new concept, named \emph{eigentuple}, about T-product tensors and they apply this concept to study nonnegative (positive) definite properties of T-product tensors. The final main contribution of this work is to develop the Courant-Fischer Theorem with respect to eigentuples, and this theorem helps us to understand the relationship between the minimum eigentuple and the maximum eigentuple. The main content of this paper is Part I of a serious task about T-product tensors. The Part II of this work will utilize these new inequalities and Courant-Fischer Theorem under T-product tensors to derive tail bounds of the extreme eigenvalue and the maximum eigentuple for sums of random T-product tensors, e.g., T-product tensor Chernoff and T-product tensor Bernstein bounds.
\end{abstract}



\newpage

\section{Introduction}\label{sec:Introduction} 

\subsection{T-product Tensors}
The T-product operation between two three order tensors was introduced by Kilmer and her collaborators in~\cite{kilmer2011factorization, kilmer2013third}. It has been shown as a powerful tool in many fields: signal processing~\cite{zhang2016exact, semerci2014tensor}, machine learning~\cite{settles2007multiple}, computer vision~\cite{zhang2014novel, martin2013order}, image processing~\cite{khalil2021efficient}, low-rank tensor approximation~\cite{xu2013parallel, zhou2017tensor, qi2021tsingular} etc. Due to wide applications of T-product, T-SVD and tubal ranks, Qi et. al~\cite{qi2021tquadratic}. extend eigentuple concept first defined from~\cite{braman2010third} to study properties for symmetry of T-product tensors and positive (nonnegative) semidefiniteness of T-product tensors by defining a T-quadratic form, whose variable is an $m \times p$ matrix, and whose value is a $p$-dimensional vector. They further show that a T-quadratic form is positive semidefinite (definite) if and only if the smallest eigentuple of the corresponding T-symmetric tensor is nonnegative (positive). Besides T-quadratic form, general functions for T-product tensors and their properties are also studied based on T-SVD, see~\cite{li2020continuity, zheng2021t, miao2021t, miao2020generalized}. 
However, none of these works discussed further issues about inequalities associated with T-product tensors. The first part of this work about T-product tensors is to develop several new T-product tensors inequalities, which are the main topics discussed by this paper. In the matrix setting, there are many useful applications about these similar inequalities under the traditional matrix product, e.g., quantum information processing~\cite{Lemm_2018}. The second part of this work is to apply these new inequalities about T-product tensors to tail bounds estimation of the maximum eigenvalue and the maximum eigentuple for sums of random T-product tensors, e.g., Chernoff and Bernstein bounds. We will introduce these new inequalities about T-product tensors obtained at this Part I work at the next subsection. 

\subsection{New Inequalities about T-product Tensors}

In this work, we define trace, denoted by $\mathrm{Tr}$, as the summation of f-diagonal entries of a given symmetric T-product tensor $\mathcal{C} \in \mathbb{C}^{m \times m \times p}$ and study properties of trace for T-product tensors. Our first main inequality about trace is following theorem:
\begin{restatable}[Monotonicity and Convexity of T-product Trace Function]{thm}{ThmConvexityMonotonicityTraceFunc}\label{thm:convexity and monotonicity of a trace func intro}
Let $f: \mathbb{R} \rightarrow \mathbb{R}$ be a continuous function with non-decreasing / convex / strictly convex properties, then so is the mapping $\mathcal{C} \rightarrow \mathrm{Tr} \left( f ( \mathcal{C})\right)$. 
\end{restatable}

From trace definition, we will prove Golden–Thompson inequality for two Hermitian T-product tensors which will be utilized to prove T-product tensors martingale inequalities.   A Hermitian T-product tensor is a tensor equal to its Hermiitan transpose, which is defined by Eq.~\eqref{eq:Hermitian Transpose Def}. 
\begin{restatable}[Golden-Thompson inequality for T-product Tensors]
{thm}{ThmGTInequalityTProduct}\label{thm:GT Inequality for T-product Tensors intro}
Given two Hermitian T-product tensors $\mathcal{C}, \mathcal{D} \in \mathbb{C}^{m \times m \times p}$, we have 
\begin{eqnarray}
\mathrm{Tr} \left( \exp(\mathcal{C} + \mathcal{D}) \right) \leq \mathrm{Tr} \left( \exp \left( \mathcal{C} \right)  \star \exp \left( \mathcal{D} \right) \right),
\end{eqnarray}
where $\star$ is the product operation between two T-product tensors defined by Eq.~\eqref{eq:T-product def}.
\end{restatable}

The next inequality we will show is Jensen's operator inequality (positive semidefinite relation between two T-product tensors).  A tensor with T-positive definite (or T-positive semi-definite) will be abbreviated as TPD (or TPSD), see Section~\ref{subsec:T-Symmetric Positive Semidefinite Tensors} for its definition. If we have a TPSD relation between two T-product tensors $\mathcal{C}$ and $\mathcal{D}$ represented as $\mathcal{C} \preceq \mathcal{D}$, then the difference T-product tensor $\left(\mathcal{D} - \mathcal{C}\right)$ is a T-positive semi-definite tensor. Let  $\mathcal{I}_{mmp} \in \mathbb{C}^{m \times m \times p}$ be the identity tensor defined by Eq.~\eqref{eq:I_mmp def}.
\begin{restatable}[Jensen's T-product Inequality]
{thm}{ThmJensenTProductInequality}\label{thm:Jensen's T-product Inequality intro}
For a continuous \emph{T-product tensor convex} function $f$ defined on an interval $\mathrm{I}$. The definition of \emph{T-product tensor convex} is given by Eq.~\eqref{eq:T tensor convex}. we have the following TPSD relation for each natural number $n$:
\begin{eqnarray}\label{eq1:thm:Jensen's T-product Inequality intro}
f \left( \sum\limits_{i=1}^{n} \mathcal{C}_i^{\mathrm{H}} \star \mathcal{X}_i \star \mathcal{C}_i \right) \preceq 
\sum\limits_{i=1}^{n} \mathcal{C}_i^{\mathrm{H}} \star f\left(   \mathcal{X}_i \right)\star  \mathcal{C}_i,
\end{eqnarray}
where $\mathcal{X}_i \in \mathbb{C}^{m \times m \times p}$ are bounded, Hermitian T-product tensors with all eigenvalues in the interval $\mathrm{I}$ and tensors $\mathcal{C}_i$ satisfying $\sum\limits_{i=1}^{n} \mathcal{C}_i^{\mathrm{H}} \star  \mathcal{C}_i  = \mathcal{I}_{mmp}$.
\end{restatable}

The immediate application of Theorem~\ref{thm:convexity and monotonicity of a trace func intro} is to prove Klein's inequality for T-product tensor. 
\begin{restatable}[Klein's T-product Inequality]
{thm}{ThmKleinsTProductInequality}
\label{thm:Klein's Ineqaulity T-product intro}
For all $\mathcal{C}, \mathcal{D}$ Hermitian T-product tensors and a differentiable convex function $f: \mathbb{R} \rightarrow \mathbb{R}$ or for all $\mathcal{C}, \mathcal{D}$ Hermitian T-product tensors and a differentiable convex function $f: (0, \infty) \rightarrow \mathbb{R}$, we have
\begin{eqnarray}
\mathrm{Tr}\left( f(\mathcal{C}) -  f(\mathcal{D}) - (\mathcal{C} - \mathcal{D}) \star f'(\mathcal{D}) \right) \geq 0. 
\end{eqnarray}
In both situations, if $f$ is strictly convex, equality holds if and only if $\mathcal{C} = \mathcal{D}$. 
\end{restatable}

Previous theorems will help us to establish the following main theorem of this paper about Lieb's concavity for T-product tensors since tail bounds for sums of random T-product tensors will be derived based on such concavity relation.

\begin{restatable}[Lieb's concavity theorem for T-product tensors]
{thm}{ThmLiebConcavityTProduct}\label{thm:Lieb concavity thm}
Let $\mathcal{H}$ be a Hermitian T-product tensor. Following map 
\begin{eqnarray}\label{eq:thm:Lieb concavity thm}
\mathcal{A} \rightarrow \mathrm{Tr}e^{\mathcal{H}+ \log \mathcal{A}}
\end{eqnarray}
is concave on the positive-definite cone.
\end{restatable}

We are ready to present the theorem for the tail bound of the maximum eigenvalue induced by independent sums of random Hermitian T-product tensors and this theorem 
will play a key role to establish various new tail bounds of the maximum eigenvalue generated by independent sums of random T-product tensors. 
\begin{restatable}[Master Tail Bound for Independent Sum of Random T-product Tensors for Eigenvalue]{thm}{ThmMasterTailBoundEigenvalue}\label{thm:Master Tail Bound for Independent Sum of Random Tensors}
Given a finite sequence of independent Hermitian T-product tensors $\{ \mathcal{X}_i \}$, we have 
\begin{eqnarray}\label{eq1:thm:master tail bound}
\mathrm{Pr} \left( \lambda_{\max} \left(\sum\limits_{i=1}^n \mathcal{X}_i \right) \geq \theta \right)
& \leq & \inf\limits_{t > 0} \Big\{ e^{- t \theta} \mathrm{Tr}\exp \left( \sum\limits_{i=1}^{n} \log \mathbb{E} e^{t \mathcal{X}_i}  \right) \Big\}. 
\end{eqnarray}
\end{restatable}

Similarly, we can generalize master tail bound for independent sum of random Hermitian T-product tensors with respect to eigenvalue from Theorem~\ref{thm:Master Tail Bound for Independent Sum of Random Tensors} to eigentuple version by the following theorem~\ref{thm:master tail bound eigentuple}. We begin with $\bigodot$ operation defined in Proposition 2.1 from work~\cite{qi2021tquadratic}. 

Let $\mathbf{a}=(a_1,a_2,\cdots,a_p)^{\mathrm{T}} \in \mathbb{C}^p$, then operator $\mbox{circ}$ to the vector $\mathbf{a}$ can be defined as
\begin{eqnarray}
\mbox{circ}(\mathbf{a}) \define 
  \left( \begin{array}{ccccc}
       a_1  &  a_p  & a_{p-1} & \cdots  &  a_2 \\
       a_2  &  a_1  & a_{p} & \cdots  &  a_3 \\
       \vdots  &  \vdots  & \vdots  & \cdots  & \vdots   \\
       a_p  &  a_{p-1}  & a_{p-2} & \cdots  &  a_1 \\
    \end{array}
\right),
\end{eqnarray}
and $\mbox{circ}^{-1}( \mbox{circ} ( \mathbf{a}))  \define \mathbf{a}$. Suppose that $\mathbf{a}, \mathbf{b} \in \mathbb{C}^p$, we define 
\begin{eqnarray}
\mathbf{a} \bigodot \mathbf{b} \define \mbox{circ}(\mathbf{a})\cdot \mathbf{b},
\end{eqnarray}
where $\cdot$ is the standard matrix and vector multiplication. Then, we are ready to present the following theorem.
\begin{restatable}[Master Tail Bound for Independent Sum of Random T-product Tensors for Eigentuple]{thm}{ThmMasterTailBoundEigentuple}\label{thm:master tail bound eigentuple}
Given a finite sequence of independent random Hermitian T-product tensors $\{ \mathcal{X}_i \}$ such that $ \mathcal{X}_i \in \mathbb{C}^{m \times m \times p}$, if $\sum\limits_{i=1}^n t \mathcal{X}_i $ satisfies Eq.~\eqref{eq1:lma: Laplace Transform Method Eigentuple Version}, we have 
\begin{eqnarray}\label{eq1:thm:master tail bound eigentuple}
\mathrm{Pr} \left( \mathbf{d}_{\max} \left( \sum\limits_{i=1}^n \mathcal{X}_i \right) \geq \mathbf{b} \right)
& \leq & \inf\limits_{t > 0} \min\limits_{1 \leq j \leq p} \left\{ \frac{ \mathrm{Tr} \exp\left( \sum\limits_{i=1}^n \log \mathbb{E}e^{t \mathcal{X}_i }  \right)     }{ \left( e_{\bigodot}^{t \mathbf{b} }\right)_j} \right\},
\end{eqnarray}
where $e_{\bigodot}^{t \mathbf{b}} \in \mathbb{C}^{p}$ is the exponential for the vector $t \mathbf{b}$ with respect to $\bigodot$ operation.
\end{restatable}

The last important theorem is the Courant-Fischer theorem for T-product tensors. This theorem will be used to figure out the relationship between the maximum eigentuple and the minimum eigentule of a T-product tensor.

\begin{restatable}[Courant-Fischer Theorem under T-product]{thm}{ThmCourantFischerTProduct}\label{thm:Courant-Fischer Theorem under T-product}
Let $\mathcal{A} \in \mathbb{C}^{m \times m \times p}$ be a Hermitian T-product tensor with eigentuples $\mathbf{d}_1 \geq \mathbf{d}_2 \geq \cdots \geq \mathbf{d}_n$. Let $\{\mathbf{U}_{j}^{[l]} \} \in \mathbb{C}^{m \times p }$ be orthnomal matrices for $1 \leq j \leq m$ and $0 \leq l \leq p-1$, $S_k$ be the space spanned by $\{\mathbf{U}_{j}^{[l]} \}$  for $1 \leq j \leq k$ and $0 \leq l \leq p-1$, and $T_k$ be the space spanned by $\{\mathbf{U}_{j}^{[l]} \}$  for $k \leq j \leq m$ and $0 \leq l \leq p-1$. Then, we have 
\begin{eqnarray}
\mathbf{d}_k &=& \max\limits_{\substack{S_k \subseteq \mathbb{C}^{m \times p} \\ \dim(S_k) = k \times p } } \min\limits_{\mathbf{X} \in S_k} \left(\mathbf{X}^{\mathrm{H}} \star \mathcal{A} \star \mathbf{X} \right) \bigg /_{\bigodot} \left( \mathbf{X}^{\mathrm{H}} \star \mathbf{X} \right) \nonumber \\
&=& \min\limits_{\substack{ T_k \subseteq \mathbb{C}^{m \times p} \\ \dim(T_k) = (m - k +1) \times p } } \max\limits_{\mathbf{X} \in T_k} \left(\mathbf{X}^{\mathrm{H}} \star \mathcal{A} \star \mathbf{X} \right) \bigg /_{\bigodot} \left( \mathbf{X}^{\mathrm{H}} \star \mathbf{X} \right),
\end{eqnarray}
where  $\bigg /_{\bigodot}$ is the division (inverse operation) under $\bigodot$. 
\end{restatable}

All these inequalities and maximum/minimum eigentuples relation about T-product tensors will be utilized to derive a serious of new tail bounds for the extreme eigenvalue and eigentuple for sums of random T-product tensors. These new inequalities different from author Chang's previous works about bounds for sums of random tensors based on Einstein product~\cite{chang2020convenient, chang2021general}. 

\subsection{Paper Organization}

The rest of this paper is organized as follows. In Section~\ref{sec:Tproduct Tensors}, basic notions of T-product tensors are introduced. Lieb's concavity theorem under T-product will be studied in Section~\ref{sec:Lieb's Concavity Under Tproduct}. General tail bounds for random T-product tensors are provided in Section~\ref{sec:Tail Bounds By Concatenation of Liebs Concavity}. Courant-Fischer Theorem under T-product is given in Section~\ref{sec:Courant-Fischer Theorem under T-product Tensors and Minimum}. Finally, conclusion will be drawn in Section~\ref{sec:Conclusion}.

\noindent \textbf{Nomenclature:} The sets of complex and real numbers are denoted by $\mathbb{C}$ and $\mathbb{R}$, respectively. The symbol $\define$ denotes mathematical definition.

\section{T-product Tensors}\label{sec:Tproduct Tensors} 

In this section, we will review T-product operations briefly and discuss related properties in Sec.~\ref{subsec:What are T-product Tensors}. The T-SVD decomposition of T-product tensors and T-Symmetric tensors will be presented in Sec.~\ref{subsec:TSVD}

\subsection{What are T-product Tensors}\label{subsec:What are T-product Tensors}

For a third order tensor $\mathcal{C} \in \mathbb{C}^{m \times n \times p}$, we define $\mbox{bcirc}$ operation to the tensor $\mathcal{C}$ as:
\begin{eqnarray}
\mbox{bcirc} (\mathcal{C} ) \define \left(
    \begin{array}{ccccc}
       \mathbf{C}^{(1)}  &  \mathbf{C}^{(p)}  &  \mathbf{C}^{(p-1)}  & \cdots  &  \mathbf{C}^{(2)}   \\
       \mathbf{C}^{(2)}  &  \mathbf{C}^{(1)}  &  \mathbf{C}^{(p)}  & \cdots  &  \mathbf{C}^{(3)}   \\
       \vdots  &  \vdots  & \vdots  & \cdots  & \vdots   \\
       \mathbf{C}^{(p)}  &  \mathbf{C}^{(p-1)}  &  \mathbf{C}^{(p-2)}  & \cdots  &  \mathbf{C}^{(1)}   \\
    \end{array}
\right),
\end{eqnarray}
where $\mathbf{C}^{(1)}, \cdots, \mathbf{C}^{(p)} \in \mathbb{C}^{ m \times n}$ are frontal slices of tensor $\mathcal{C}$. The inverse operation of $\mbox{bcirc}$ is denoted as $\mbox{bcirc}^{-1}$ with relation $\mbox{bcirc}^{-1} ( \mbox{bcirc} ( \mathcal{C} )) \define \mathcal{C}$. 

For a third order tensor $\mathcal{C} \in \mathbb{C}^{m \times m \times p}$, we define Hermitian transpose of $\mathcal{C}$, denoted by $\mathcal{C}^{\mathrm{H}}$ , as 
\begin{eqnarray}\label{eq:Hermitian Transpose Def}
\mathcal{C}^{\mathrm{H}} = \mbox{bcirc}^{-1}(     (\mbox{bcirc}(\mathcal{C}))^{\mathrm{H}}  ). 
\end{eqnarray}
And a tensor $\mathcal{D} \in \mathbb{C}^{m \times m \times p }$ is called a Hermitian T-product tensor if $ \mathcal{D}^{\mathrm{H}}  = \mathcal{D}$. Similarly, 
we also define ranspose of $\mathcal{C}$, denoted by $\mathcal{C}^{\mathrm{T}}$ , as 
\begin{eqnarray}\label{eq:Only Transpose Def}
\mathcal{C}^{\mathrm{T}} = \mbox{bcirc}^{-1}(     (\mbox{bcirc}(\mathcal{C}))^{\mathrm{T}}  ). 
\end{eqnarray}
And a tensor $\mathcal{D} \in \mathbb{C}^{m \times m \times p }$ is called a Symmetric T-product tensor if $ \mathcal{D}^{\mathrm{T}}  = \mathcal{D}$.

The identity tensor $\mathcal{I}_{mmp} \in \mathbb{C}^{m \times m \times p }$ can be defined as:
\begin{eqnarray}\label{eq:I_mmp def}
\mathcal{I}_{mmp} = \mbox{bcirc}^{-1}(  \mathbf{I}_{mp}  ),
\end{eqnarray}
where $  \mathbf{I}_{mp} $ is the identity matrix in $\mathbb{R}^{mp \times mp}$. A zero tensor, denoted as $\mathcal{O}_{mnp} \in \mathbb{C}^{m \times n \times p}$, is a tensor that all elements inside the tensor as $0$. 

In order to define the T-product operation, we need to define another kind of operation over a third order tensor. For a third order tensor $\mathcal{C} \in \mathbb{C}^{m \times n \times p}$, we define $\mbox{unfold}$ operation to the tensor $\mathcal{C}$ as:
\begin{eqnarray}
\mbox{unfold} (\mathcal{C} ) \define \left(
    \begin{array}{c}
     \mathbf{C}^{(1)} \\
     \mathbf{C}^{(2)} \\
      \vdots \\
     \mathbf{C}^{(p)} \\
    \end{array}
\right),
\end{eqnarray}
where $\mbox{unfold} (\mathcal{C} ) \in \mathbb{C}^{mp \times n}$, and the inverse operation of $\mbox{unfold}$ is $\mbox{fold}$ with the relation $\mbox{fold}(\mbox{unfold} ( \mathcal{C} ) ) \define \mathcal{C}$. Given  $\mathcal{C} \in \mathbb{C}^{m \times n \times p}$ and $\mathcal{D} \in \mathbb{C}^{n \times k \times p}$, we define the T-product between  $\mathcal{C}$ and  $\mathcal{D}$ as 
\begin{eqnarray}\label{eq:T-product def}
\mathcal{C} \star \mathcal{D} \define \mbox{fold} ( \mbox{bcirc}( \mathcal{D}) \mbox{unfold}(\mathcal{D}) ), 
\end{eqnarray}
where $\mathcal{C} \star \mathcal{D} \in \mathbb{C}^{m \times k \times p}$. 

\begin{definition}\label{def:standard form of a real f-Diagonal Tensor}
Let $\mathcal{S} = (s_{ijk}) \in \mathbb{C}^{m \times n \times p}$ be a f-diagonal tensor, i.e., each frontal slice of tensor $\mathcal{S}$ is a diagonal matrix. Let $\mathbf{s}_i = (s_{ii1}, s_{ii2}, \cdots, s_{iip})^{\mathrm{T}}$ be the $ii-$th tube of $\mathcal{S}$ for $1 \leq i \leq \min\{m,n\}$. The  f-diagonal tensor $\mathcal{S}$ is in its \emph{standard form} if $\mathbf{s}_1 \geq \mathbf{s}_2 \geq \cdots \geq \mathbf{s}_{\min\{m,n\}}$, where $\geq$ is the elementwise comparison between two vectors.
\end{definition}

% trace def

We define the T-product tensor \emph{trace} for a tensor $\mathcal{C} =  (c_{ijk}) \in \mathbb{C}^{m \times m \times p}$, denoted by $\mathrm{Tr}(\mathcal{C})$, as following
\begin{eqnarray}\label{eq:trace def}
\mathrm{Tr}(\mathcal{C}) \define \sum\limits_{i=1}^{m}\sum\limits_{k=1}^{p} c_{iik},
\end{eqnarray}
which is the summation of all entries in f-diagonal components. Then, we have following lemma about trace properties. 
% trace (AB) = trace (BA)
\begin{lemma}\label{lma:T product trace properties}
For any tensors $\mathcal{C}, \mathcal{D} \in \mathbb{C}^{m \times m \times p}$, we have 
\begin{eqnarray}\label{eq:trace linearity}
\mathrm{Tr}(c \mathcal{C} + d \mathcal{D}) = c  \mathrm{Tr}(\mathcal{C}) + d \mathrm{Tr}(\mathcal{D}),
\end{eqnarray}
where $c, d$ are two contants. And, the transpose operation will keep the same trace value, i.e., 
\begin{eqnarray}\label{eq:trace transpose same}
\mathrm{Tr}(\mathcal{C}) = \mathrm{Tr}(\mathcal{C}^{\mathrm{T}}).
\end{eqnarray}
Finally, we have 
\begin{eqnarray}\label{eq:trace commutativity}
\mathrm{Tr}(\mathcal{C} \star  \mathcal{D} ) = \mathrm{Tr}(\mathcal{D} \star  \mathcal{C}).
\end{eqnarray}
\end{lemma}
\textbf{Proof:}
Eqs.~\eqref{eq:trace linearity} and~\eqref{eq:trace transpose same} are true from trace definiton directly. 

From Eq.~\eqref{eq:T-product def}, the $i$-th frontal slice matrix of $\mathcal{D} \star  \mathcal{C}$ is 
\begin{eqnarray}\label{eq:i slice DC}
\mathbf{D}^{(i)} \mathbf{C}^{(1)}  + \mathbf{D}^{(i-1)} \mathbf{C}^{(2)} + \cdots +
\mathbf{D}^{(1)} \mathbf{C}^{(i)} + \mathbf{D}^{(m)} \mathbf{C}^{(i+1) } + \cdots
+ \mathbf{D}^{ (i+1)} \mathbf{C}^{(m)}, 
\end{eqnarray}
similarly, the $i$-th frontal slice matrix of $\mathcal{C} \star  \mathcal{D}$ is 
\begin{eqnarray}\label{eq:i slice CD}
\mathbf{C}^{(i)} \mathbf{D}^{(1)}  + \mathbf{C}^{(i-1)} \mathbf{D}^{(2)} + \cdots +
\mathbf{C}^{(1)} \mathbf{D}^{(i)} + \mathbf{C}^{(m)} \mathbf{D}^{(i+1)} + \cdots
+ \mathbf{C}^{ (i+1)}  \mathbf{D}^{(m)}. 
\end{eqnarray}
Because the matrix trace of Eq.~\eqref{eq:i slice DC} and the matrix trace of Eq.~\eqref{eq:i slice CD} are same for each slice $i$ due to linearity and invariant under cyclic permutations of matrix trace, we have Eq.~\eqref{eq:trace commutativity} by summing over all frontal matrix slices. $\hfill \Box$

% T-product determinant

Below, we will define the \emph{determinant} of a T-product tensor $\mathcal{C} \in \mathbb{C}^{m \times m \times p}$ and its asscoiate properties. The determinant of a $m \times m \times p$ tensor $\mathcal{C}$ is the $m$-linear alternating form defined as
\begin{eqnarray}\label{eq:det def}
\mathrm{det}: \left(\mathbf{V}_1, \cdots, \mathbf{V}_m \right) \rightarrow \mathbb{C}, 
\end{eqnarray}
where $\mathbf{V}_i  \in \mathbb{C}^{m \times p}$ is the $i$-th lateral matrix of the tensor $\mathcal{C}$. Moreover, we require that $\mathrm{det}(  \mathcal{I}_{mmp}  ) = 1$. Given two tensors  $\mathcal{C}, \mathcal{D} \in \mathbb{C}^{m \times m \times p}$, the determinant of $\mathcal{C} \star \mathcal{D}$ is $\mathrm{det}( \mathcal{C} \star \mathcal{D}  ) = \lambda \mathrm{det}( \mathcal{D} ) $ for some value $\lambda$. If we set $\mathcal{D}$ as $\mathcal{I}_{mmp}$, we have 
\begin{eqnarray}
\mathrm{det}( \mathcal{C} \star \mathcal{I}_{mmp}  )  = \lambda \mathrm{det}( \mathcal{I}_{mmp} ) = \lambda = \mathrm{det}( \mathcal{C} ).
\end{eqnarray}
Then, we have 
\begin{eqnarray}\label{eq:homo of T-tensors det}
\mathrm{det}( \mathcal{C} \star \mathcal{D}  ) = \mathrm{det}( \mathcal{C}) \mathrm{det}( \mathcal{D}  ) 
\end{eqnarray}

\subsection{T-SVD Decomposition}\label{subsec:TSVD}

Given a tensor $\mathcal{C} \in \mathbb{C}^{m \times n \times p}$, Theorem 4.1 in~\cite{kilmer2011factorization} proposed a T-singular value decomposition (T-SVD) for $\mathcal{C}$ as:
\begin{eqnarray}\label{eq:T-SVD Thm}
\mathcal{C} = \mathcal{U} \star \mathcal{S} \star \mathcal{V}^{\mathrm{T}},
\end{eqnarray}
where $\mathcal{U} \in \mathbb{C}^{m \times m \times p}$ and $\mathcal{V} \in \mathbb{C}^{n \times n \times p}$ are orthogonal tensors, and $\mathcal{S} \in \mathbb{C}^{m \times n \times p}$ is a f-diagonal tensor. We also have $\mathcal{U}^{\mathrm{T}} \star \mathcal{U} = 
\mathcal{I}_{mmp}$ and $\mathcal{V}^{\mathrm{T}} \star \mathcal{V} = 
\mathcal{I}_{nnp}$. We define $\sigma(\mathcal{C})$ be the spectrum of $\mathcal{C}$, i.e., the set of $s \in \mathbb{C}$, where $s$ are nonzero entries from tensor $\mathcal{S}$. We use $\left\Vert \cdot \right\Vert$ for the spectral norm, which is the largest singular value of a T-product tensor. 

Given any integer $k$ and $\mathcal{B} \in \mathbb{C}^{m \times m \times p}$, we define $\mathcal{B}^k$ as 
\begin{eqnarray}
\mathcal{B}^k \define \overbrace{\mathcal{B} \star  \mathcal{B}  \star  \mathcal{B}  \star \cdots \star  \mathcal{B} }^{\mbox{$k$ terms of $\mathcal{B}$ under T-product}}
\end{eqnarray}
where $\mathcal{B}^k \in \mathbb{C}^{m \times m \times p}$. Then, we have following corollary from T-SVD in Eq.~\eqref{eq:T-SVD Thm}. 
\begin{corollary}\label{cor:tensor power expression}
Suppose $\mathcal{B} \in \mathbb{C}^{m \times m \times p}$ is a Hermitian T-product tensor, and $\mathcal{S}^{-1}$ exists, where $\mathcal{S}$ is f-diagonal tensor obtained from the T-SVD of the tensor $\mathcal{C}$. Then, we have 
\begin{eqnarray}
\mathcal{B}^k =  \mathcal{U} \star \mathcal{S}^k \star \mathcal{U}^{\mathrm{T}}.
\end{eqnarray}
\end{corollary}

Then, we can define the T-product tensor exponential function and the T-product tensor logarithm function under T-product as below with tensor power.
%  exp, log,
\begin{definition}\label{def: tensor exponential}
Given a tensor $\mathcal{X} \in \mathbb{C}^{m \times m \times p}$, the \emph{tensor exponential} of the tensor $\mathcal{X}$ is defined as 
\begin{eqnarray}\label{eq: tensor exp def}
e^{\mathcal{X}} \define \sum\limits_{k=0}^{\infty} \frac{\mathcal{X}^{k}}{k !}, 
\end{eqnarray}
where $\mathcal{X}^0$ is defined as the identity tensor $\mathcal{I}_{mmp}$. Given a tensor $\mathcal{Y}$, the tensor $\mathcal{X}$ is said to be a \emph{tensor logarithm} of $\mathcal{Y}$ if $e^{\mathcal{X}}  = \mathcal{Y}$.
\end{definition}

% spectrum mapping theorem 

From T-SVD in Eq.~\eqref{eq:T-SVD Thm}, we can express a Hermitian T-product tensor $\mathcal{C} \in \mathbb{C}^{m \times m \times p}$ as 
\begin{eqnarray}\label{eq:spectral expression}
\mathcal{C} = \sum\limits_{i=1}^{m}\sum\limits_{k=0}^{p-1} s_{iik} \mathbf{U}_i^{[k]} \star \left(\mathbf{U}_i^{[k]} \right)^{\mathrm{T}}, 
\end{eqnarray}
where $ s_{iik} $ are \emph{eigenvalues} of the tensor $\mathcal{C}$, and $ \mathbf{U}_i^{[k]} \in \mathbb{C}^{m \times 1 \times p}$ is the $i$-th lateral slice (matrix) of the tensor $\mathcal{U}$ after $k$ cyclic permutations.  The matrix $\mathbf{U}_i^{[0]}$ is obtained from the $i$-th lateral slice (matrix) of the tensor $\mathcal{U}$ with column vectors as $\mathbf{u}_i^{(1)}, \cdots, \mathbf{u}_i^{(p)}$, then we have 
\begin{eqnarray}
\mathbf{U}_i^{[k]} = \left( \mathbf{u}_i^{ (p+1 - k) \bmod p}, \mathbf{u}_i^{(p+2 - k) \bmod p }, \cdots, 
 \mathbf{u}_i^{(p)}, \mathbf{u}_i^{(1)},  \cdots   \mathbf{u}_i^{(p - k)}   \right). 
\end{eqnarray}
Note that we have $\left( \mathbf{U}_i^{[k]} \right)^{\mathrm{H}} \star \mathbf{U}_i^{[k]} = \mathcal{I}_{11p}$ and  $\left( \mathbf{U}_i^{[k]} \right)^{\mathrm{H}} \star \mathbf{U}_{i'}^{[k']} = \mathcal{O}_{11p}$ for $i \neq i'$ or $k \neq k'$. From Theorem 3.6 in~\cite{qi2021tquadratic}, all values of $s_{iik}$ are real and we define $\lambda_{\max} \define \max\limits_{\substack{1 \leq i \leq m \\ 0 \leq k \leq p-1} } \{ s_{iik} \}$, and  $\lambda_{\min} \define \min\limits_{\substack{1 \leq i \leq m \\ 0 \leq k \leq p-1} } \{ s_{iik} \}$.

From Corollary~\ref{cor:tensor power expression} and Eq.~\eqref{eq:spectral expression}, we can have following spectral mapping lemma.
\begin{lemma}\label{lma:spectral mapping}
For any continous function $f: \mathbb{R} \rightarrow \mathbb{R}$ and any Hermitian T-product tensor $\mathcal{C}$, we have 
\begin{eqnarray}
f( \mathcal{C} ) =  \sum\limits_{i=1}^{m}\sum\limits_{k=0}^{p-1} f(s_{iik}) \mathbf{U}_i^{[k]} \star \left(\mathbf{U}_i^{[k]} \right)^{\mathrm{T}}.
\end{eqnarray}
\end{lemma}

\subsection{Positive Semidefinite T-product Tensors}\label{subsec:T-Symmetric Positive Semidefinite Tensors}

Given a Hermitian T-product tensor $\mathcal{C} \in \mathbb{C}^{m \times m \times p}$, and a tensor $\mathcal{X} \in \mathbb{C}^{m \times 1 \times p }$ obtained from treating the matrix $\mathbf{X} \in \mathbb{C}^{m \times p }$ as a tensor with dimensions $\mathbb{R}^{m \times 1 \times p }$. We define following quadratic form with respect to the matrix $\mathbf{X}$ as 
\begin{eqnarray}
F_{\mathcal{C}} (\mathbf{X}) \define \mathcal{X}^{\mathrm{T}} \star \mathcal{C} \star \mathcal{X}, 
\end{eqnarray}
and we say that a tensor $\mathcal{C}$ is T-positive definite (TPD) (or T-positive semi-definite (TPSD)) if $F_{\mathcal{C}} (\mathbf{X}) > \mathbf{0}$ (or $F_{\mathcal{C}} (\mathbf{X}) \geq \mathbf{0}$ )  for any $\mathbf{X} \in  \mathbb{C}^{m \times p}$, where $\mathbf{0}$ is a zero vector with size $p$. 

We now define eigentuples and eigenmatrices of a Hermitian T-product tensor which will be used to characterize TPD or TPSD for a given tensor. For a matrix $\mathbf{X} \in \mathbb{C}^{m \times p} = (\mathbf{x}^{(1)}, \cdots, \mathbf{x}^{(p)})$ , we define unfolding opeartion with respect to the matrix $\mathbf{X}$ columns, denoted by $\mbox{cunfold}(\mathbf{X})$, as 
\begin{eqnarray}\label{eq:columns unfold}
\mbox{cunfold}(\mathbf{X}) \define   \cdot  \left( \begin{array}{c}
       \mathbf{x}^{(1)} \\
       \mathbf{x}^{(2)} \\
       \vdots    \\
       \mathbf{x}^{(p)} \\
    \end{array}
\right),
\end{eqnarray}
where $\mbox{cunfold}(\mathbf{X}) \in  \mathbb{C}^{m p}$. Then, suppose that $\mathcal{C} \in \mathbb{C}^{m \times m \times p}$ is a Hermitian T-product tensor, we define $\mathcal{C} \star \mathbf{X}$ as 
\begin{eqnarray}
\mathcal{C} \star \mathbf{X} = \mbox{fold} ( \mbox{bcirc}( \mathcal{C}) \mbox{cunfold}(\mathbf{X})  ), 
\end{eqnarray}
where $\mathcal{C} \star \mathbf{X} \in  \mathbb{C}^{m \times p}$. We also define a new product operation between a vector $\mathbf{d} = (d_1, d_2, \cdots, d_p)^{\mathrm{T}}$ and a matrix $\mathbf{X} \in \mathbb{C}^{m \times p}$, denoted by $\circ$, as
\begin{eqnarray}\label{eq:circ prod def}
\mathbf{d} \circ \mathbf{X} \define \mathbf{X}   \cdot  \left( \begin{array}{ccccc}
       d_1  &  d_p  &  d_{p-1}  & \cdots  &  d_2   \\
       d_2  &  d_1  &  d_{p}  & \cdots  &  d_3   \\
       \vdots  &  \vdots  & \vdots  & \cdots  & \vdots   \\
       d_p  &  d_{p-1} &  d_{p-2}  & \cdots  &  d_1   \\
    \end{array}
\right),
\end{eqnarray}
where $\cdot$ is the standard matrix multiplication. Suppose that $\mathbf{X} \in \mathbb{C}^{m \times p}$ and $\mathbf{X} \neq \mathbf{O}$, and $\mathbf{d} \in \mathbb{C}^p$, if we have 
\begin{eqnarray}\label{eq:eigen def}
\mathcal{C} \star \mathbf{X} = \mathbf{d} \circ \mathbf{X},
\end{eqnarray}
we call $\mathbf{d}$ as an eigentuple of $\mathcal{C}$, and $\mathbf{X}$ as an eigenmatrix of $\mathcal{C}$ corresponding to the eigentuple $\mathbf{d}$.

From Theorem 4.1 in~\cite{qi2021tquadratic}, a T-square tensor $\mathcal{C} \in \mathbb{C}^{m \times m \times p}$ with eigentuples arranged as f-diagonal tensor $\mathcal{S}$ according to the standard form provided by Definition~\ref{def:standard form of a real f-Diagonal Tensor}, i.e., $\mathbf{s}_1 \geq \mathbf{s}_2 \geq \cdots \geq \mathbf{s}_m$. Then $\mathcal{C}$ is  TPD (or TPSD) if and only if the smallest eigentuple $\mathbf{s}_m > (\mbox{or $\geq$ })\mathbf{0}$.  We use $\left\Vert \mathcal{C} \right\Vert_{\mbox{\tiny{vec}}}$ to represent the spectral norm of eigentuple of the tensor $\mathcal{C}$, which is defined as 
\begin{eqnarray}
\left\Vert \mathcal{C} \right\Vert_{\mbox{\tiny{vec}}} \define \mathbf{d}_{\max}\left( \sqrt{ \mathcal{C}^{\mathrm{H}}\star \mathcal{C}  } \right).
\end{eqnarray}

\subsection{T-product Tensors Analysis}\label{subsec:T-product Tensors Analysis}

We will begin with monotonicity and convexity discussions of the trace function. 
\begin{lemma}\label{lma:monotonicity of trace func}
For a given continous and non-decreasing function $f: \mathbb{R} \rightarrow \mathbb{R}$, the associated trace function on a Hermitian T-product tensor $\mathcal{C}$  is given by 
\begin{eqnarray}
\mathcal{C} \rightarrow \mathrm{Tr} \left(  f (\mathcal{C} )\right).
\end{eqnarray}
Then we have 
\begin{eqnarray}
\mathcal{C} \succeq \mathcal{D} \Longrightarrow \mathrm{Tr} \left(  f (\mathcal{C} )\right) \geq \mathrm{Tr} \left(  f (\mathcal{D} )\right).
\end{eqnarray}
\end{lemma}
\textbf{Proof:}
We first assume that the function $f$ is differentiable, then the first derivative of $f$ is greater or equal than zero (monotonicity). We further define a trace function $g(t) \define \mathrm{Tr} \left( f \left( \mathcal{D}  + t ( \mathcal{C} - \mathcal{D}) \right) \right)$. Then, we have
\begin{eqnarray}
\mathrm{Tr} \left(  f (\mathcal{C} )\right) - \mathrm{Tr} \left(  f (\mathcal{D} )\right) = g(1) -g(0) = \int\limits_0^1 g'(t) dt = \int\limits_0^1  \mathrm{Tr} \left(  f' (\mathcal{D} + t  ( \mathcal{C} - \mathcal{D}))\star ( \mathcal{C} - \mathcal{D} ) \right) dt \nonumber \\
= \int\limits_0^1  \mathrm{Tr} \left(  ( \mathcal{C} - \mathcal{D} )^{1/2} \star f' (\mathcal{D} + t  ( \mathcal{C} - \mathcal{D}) ) \star ( \mathcal{C} - \mathcal{D} )^{1/2} \right) \geq 0,  ~~~~~~~~~~~ 
\end{eqnarray}
where we apply Lemma~\ref{lma:T product trace properties} at the last equality, and the last inequality comes from the nonnegative of $f'$. By applying the standard continuity argument, we can relax the requirement that $f$ is continuously differentiable to the requirement that $f$ is  continuous.
$\hfill \Box$

The next lemma will be used to show the convexity of trace function on a Hermitian T-product tensor $\mathcal{C}$. 
\begin{lemma}\label{lma:Peierls Inequality}
Let $\mathcal{C} \in \mathbb{C}^{m \times m \times p}$ be a Hermitian T-product tensor, $f$ convex on $\mathbb{R}$, and $\mathbf{V}_i^{[k]}$ for $1 \leq i \leq m$ and $0 \leq k \leq p-1$ be any orthnormal base of $\mathbb{C}^{m \times p}$. Then, we have 
\begin{eqnarray}\label{eq1:lma:Peierls Inequality}
\mathrm{Tr} \left( f ( \mathcal{C} ) \right) \geq \sum\limits_{i=1}^{m} \sum\limits_{k=0}^{ p -1}  f \left( \left \langle \mathbf{V}_i^{[k]}, \mathcal{C} \star \mathbf{V}_i^{[k]} \right \rangle \right), 
\end{eqnarray}
where $\left\langle \mathbf{V}_i^{[k]}, \mathcal{C} \star \mathbf{V}_i^{[k]} \right \rangle$ is the Frobenius inner product between two matrices $ \mathbf{V}_i^{[k]} $ and $ \mathcal{C} \star \mathbf{V}_i^{[k]} $. There is an equality if each  $\mathbf{V}_i^{[k]}$  is an eigenmatrix of $\mathcal{C}$ and it's the only case if $f$ is strictly convex.
\end{lemma}
\textbf{Proof:}
From spectral representation by Eq.~\eqref{eq:spectral expression}, we have 
\begin{eqnarray}
\mathrm{Tr} \left( f ( \mathcal{C} ) \right) &= & \sum\limits_{i=1}^{m} \sum\limits_{k=0}^{ p -1}
\left\langle  \mathbf{V}_i^{[k]} ,    \sum\limits_{i'=1}^{m} \sum\limits_{k'=0}^{ p -1} f( s_{i'i'k'}) \left(  \mathbf{U}_{i'}^{[k']} \star \left(\mathbf{U}_{i'}^{[k']} \right)^{\mathrm{T}} \right)      \star \mathbf{V}_i^{[k]}            \right\rangle \nonumber \\
&=&  \sum\limits_{i=1}^{m} \sum\limits_{k=0}^{ p -1}  \sum\limits_{i'=1}^{m} \sum\limits_{k'=0}^{ p -1}  f( s_{i'i'k'}) \left\Vert   \left(  \mathbf{U}_{i'}^{[k']} \star \left(\mathbf{U}_{i'}^{[k']} \right)^{\mathrm{T}} \right)      \star  \mathbf{V}_i^{[k]}  \right\Vert^2 \nonumber \\
& \geq &  \sum\limits_{i=1}^{m} \sum\limits_{k=0}^{ p -1} f \left(  \sum\limits_{i'=1}^{m} \sum\limits_{k'=0}^{ p -1}  s_{i'i'k'}  \left\langle  \mathbf{V}_{i}^{[k]},  \left(  \mathbf{U}_{i'}^{[k']} \star \left(\mathbf{U}_{i'}^{[k']} \right)^{\mathrm{T}} \right)    \star  \mathbf{V}_{i}^{[k]} \right\rangle    \right)  \nonumber \\
& = &  \sum\limits_{i=1}^{m} \sum\limits_{k=0}^{ p -1} f \left( \left\langle \mathbf{V}_{i}^{[k]}  ,  \mathcal{C} \star \mathbf{V}_{i}^{[k]} \right\rangle \right), 
\end{eqnarray}
where the only inequality comes from the convexity of the function $f$. Since for each $i, k$, we have $ \sum\limits_{i'=1}^{m} \sum\limits_{k'=0}^{ p -1}  \left\Vert   \left(  \mathbf{U}_{i'}^{[k']} \star \left(\mathbf{U}_{i'}^{[k']} \right)^{\mathrm{T}} \right)      \star  \mathbf{V}_i^{[k]}  \right\Vert^2 =  \left \Vert  \mathbf{V}_i^{[k]}  \right\Vert^2 = 1$. Note that each $ \mathbf{V}_i^{[k]} $ is an eigenmatrix of $\mathcal{C}$ if and only if $  \left\Vert   \left(  \mathbf{U}_{i'}^{[k']} \star \left(\mathbf{U}_{i'}^{[k']} \right)^{\mathrm{T}} \right)      \star  \mathbf{V}_i^{[k]}  \right\Vert^2   = 1$ for some $i', k'$, and is 0 otherwise, in which case the inequality in Eq.~\eqref{eq1:lma:Peierls Inequality} is eqaulity. When $f$ is strictly convex, equality in Eq.~\eqref{eq1:lma:Peierls Inequality} can be true only if for each $i, k$, we have  $  \left\Vert   \left(  \mathbf{U}_{i'}^{[k']} \star \left(\mathbf{U}_{i'}^{[k']} \right)^{\mathrm{T}} \right)      \star  \mathbf{V}_i^{[k]}  \right\Vert^2   = 1$ for some $i', k'$, and is 0 otherwise. 
$\hfill \Box$

From Lemma~\ref{lma:monotonicity of trace func} and Lemma~\ref{lma:Peierls Inequality}, we have the following theorem about convexity and monotonicity of a trace function.  We recall~\cref{thm:convexity and monotonicity of a trace func intro}.

\ThmConvexityMonotonicityTraceFunc*
%
%\begin{theorem}\label{thm:convexity and monotonicity of a trace func}
%%\begin{oneshot}{Theorem~\ref{thm:convexity and monotonicity of a trace func intro}}
%%
%Let $f: \mathbb{R} \rightarrow \mathbb{R}$ be continuous function with non-decreasing / convex / strictly convex proerties, then so is the mapping $\mathcal{C} \rightarrow \mathrm{Tr} \left( f ( \mathcal{C})\right)$. 
%%
%%\end{oneshot}
%\end{theorem}
%%
\textbf{Proof:}
Given two $\mathcal{C}, \mathcal{D} \in \mathbb{C}^{m \times m \times p}$ Hermitian T-product tensors, $f$ as a convex function, and $\mathbf{V}_i^{[k]}$ for $1 \leq i \leq m$ and $0 \leq k \leq p-1$ be an orthonormal basis of $\mathbb{C}^{m \times p}$ consisting of eigenmatrices of $\frac{\mathcal{C} + \mathcal{D}}{2}$. Then, from Lemma~\ref{lma:Peierls Inequality}, we have
\begin{eqnarray}\label{eq1:thm:convexity and monotonicity of a trace func}
\mathrm{Tr}\left(f ( \frac{\mathcal{C} + \mathcal{D}}{2}  ) \right) & =  & 
\sum\limits_{i=1}^{m} \sum\limits_{k=0}^{ p -1} f \left( \left\langle \mathbf{V}_i^{[k]} , \frac{\mathcal{C} + \mathcal{D}}{2} \star \mathbf{V}_i^{[k]} \right\rangle \right) \nonumber \\
&=& \sum\limits_{i=1}^{m} \sum\limits_{k=0}^{ p -1} f \left( \frac{1}{2} \left\langle \mathbf{V}_i^{[k]}, \mathcal{C} \star \mathbf{V}_i^{[k]}\right\rangle +
\frac{1}{2} \left\langle \mathbf{V}_i^{[k]}, \mathcal{D} \star \mathbf{V}_i^{[k]}\right\rangle \right) \nonumber \\
& \leq &  \sum\limits_{i=1}^{m} \sum\limits_{k=0}^{ p -1}  \left( \frac{1}{2} f \left( \left\langle \mathbf{V}_i^{[k]}, \mathcal{C} \star \mathbf{V}_i^{[k]}\right\rangle \right) +
\frac{1}{2} f \left(\left\langle \mathbf{V}_i^{[k]}, \mathcal{D} \star \mathbf{V}_i^{[k]}\right\rangle  \right) \right) \\
& \leq &  \frac{1}{2} \mathrm{Tr} \left( f \left( \mathcal{C} \right)\right) + \frac{1}{2} \mathrm{Tr} \left( f \left( \mathcal{D} \right)\right)
\end{eqnarray}
where inequalities come from Lemma~\ref{lma:Peierls Inequality}. This demonstrates that the map $\mathcal{C} \rightarrow \mathrm{Tr}\left(f \left( \mathcal{C} \right) \right) $ is midpoint convex. 

For the strict convexity of $f$ and $\mathrm{Tr}\left(f ( \frac{\mathcal{C} + \mathcal{D}}{2}  ) \right)  =  \frac{1}{2} \mathrm{Tr} \left( f \left( \mathcal{C} \right)\right) + \frac{1}{2} \mathrm{Tr} \left( f \left( \mathcal{D} \right)\right) $, we have $\left\langle \mathbf{V}_i^{[k]}, \mathcal{C} \star \mathbf{V}_i^{[k]} \right\rangle =\left\langle \mathbf{V}_i^{[k]}, \mathcal{D} \star \mathbf{V}_i^{[k]} \right\rangle  $ for each $\mathbf{V}_i^{[k]} $. From  Lemma~\ref{lma:Peierls Inequality}, the equality will be true when  $\mathbf{V}_i^{[k]} $ are eigenmatrices for both tensors $\mathcal{C}$ and $\mathcal{D}$. Then, we have
\begin{eqnarray}
\mathcal{C} \star \mathbf{V}_i^{[k]}  = \langle  \mathbf{V}_i^{[k]}  ,  \mathcal{C} \star \mathbf{V}_i^{[k]}\rangle \mathbf{V}_i^{[k]} = \langle \mathbf{V}_i^{[k]} , \mathcal{D} \star \mathbf{V}_i^{[k]}\rangle \mathbf{V}_i^{[k]} = \mathcal{D} \star \mathbf{V}_i^{[k]}, 
\end{eqnarray}
which indicates that $\mathcal{C} = \mathcal{D}$. An obvious continuity argument now shows that if $f$ continuous as well as convex, $\mathcal{C} \rightarrow \mathrm{Tr} \left( f ( \mathcal{C})\right)$ is convex, and strictly convex so if $f$ is strictly convex. Therefore, this Theorem is proved from Lemma~\ref{lma:monotonicity of trace func} and above arguments. 
$\hfill \Box$

From T-SVD, we have following relation for Hermitian T-product tensor:
\begin{eqnarray}\label{eq:transfer rule}
f(s) \leq g(s)~~ \mbox{for $s \in [a, b]$} \Longrightarrow f(\mathcal{C}) \preceq  g(\mathcal{C}) \mbox{~~when the eigenvalues of $\mathcal{C}$ lie in $[a, b]$.}
\end{eqnarray}
Above Eq.~\eqref{eq:transfer rule} is named as transfer rule. 

We have defined tensor exponential under Definition~\ref{def: tensor exponential}, and the exponential of an Hermitian T-product tensor is always TPD due to the spectral mapping Lemma~\ref{lma:spectral mapping}. From transfer rule Eq.~\eqref{eq:transfer rule}, the tensor exponential satisfies following relations for a Hermitian T-product tensor $\mathcal{C} \in \mathbb{C}^{m \times m \times p}$ that we will use at later theory development:
\begin{eqnarray}\label{eq:1 + a leq e a}
\mathcal{I}_{mmp} + \mathcal{C} \preceq \exp (\mathcal{C}),
\end{eqnarray}
and
\begin{eqnarray}\label{eq:cosh a leq e a2/2}
\cosh ( \mathcal{C} )  \preceq \exp (\mathcal{C}^2 / 2).
\end{eqnarray}

From Theorem~\ref{thm:convexity and monotonicity of a trace func intro} and the monotonicity of the exponential function, we have 
\begin{eqnarray}\label{eq:trace exp function}
\mathcal{C} \preceq \mathcal{D} \Longrightarrow \mathrm{Tr} \exp (\mathcal{C}  ) \leq  \mathrm{Tr} \exp (\mathcal{D})
\end{eqnarray}

Below, we want to prove the monotonicity and the concavity of the logarithm function. We will begin definitions about T-product tensor monotonicity and convexity first and present several lemmas used to establish the monotonicity and the concavity of the logarithm function. Given two Hermitian T-product tensors $\mathcal{C}, \mathcal{D} \in \mathbb{C}^{m \times m \times p}$, a function $f : \mathbb{R} \rightarrow \mathbb{R}$ is said to be T-product tensor monotone if the following relation holds:
\begin{eqnarray}\label{eq:T tensor monotone}
\mathcal{C} \preceq \mathcal{D} \Longrightarrow f(\mathcal{C}) \preceq f (\mathcal{D}).
\end{eqnarray}
A function $f$ is said as a \emph{T-product tensor convex} function if we have:
\begin{eqnarray}\label{eq:T tensor convex}
f(t \mathcal{C} + (1-t) \mathcal{D}) \preceq  t f(\mathcal{C}) + (1-t) f(\mathcal{D}) (\mathcal{D}),
\end{eqnarray}
where $0 \leq t \leq 1$. Also, a function $f$ is said as a \emph{T-product tensor concave} function if $-f$ is a \emph{T-product tensor convex} function. The following derivation about the monotonicity and the concavity of the logarithm function is extended from matrices according to works in~\cite{hiai2010matrix} and~\cite{chansangiam2015survey} to T-product tensors.

\begin{lemma}\label{lma:1.5.7 Hiai paper 2010}
For any $\mathcal{C}, \mathcal{D} \in \mathbb{C}^{m \times m \times p}$, we have 
$\sigma(\mathcal{C} \star \mathcal{D}) = \sigma(\mathcal{D} \star \mathcal{C})$. 
\end{lemma}
\textbf{Proof:}
Since eigenvalues are roots of the characteristic polynomial, it is enough to show that $\mathrm{det} (\lambda \mathcal{I}_{mmp} - \mathcal{C} \star \mathcal{D}) = 
\mathrm{det} (\lambda \mathcal{I}_{mmp} - \mathcal{D} \star \mathcal{C}) $. We first assume that $\mathcal{C}$ has inverse, then from Eq.~\eqref{eq:homo of T-tensors det}, we have 
\begin{eqnarray}
\mathrm{det} (\lambda \mathcal{I}_{mmp} - \mathcal{C} \star \mathcal{D} )
= \mathrm{det} (\mathcal{C}^{-1} \star  (\lambda \mathcal{I}_{mmp} - \mathcal{C} \star \mathcal{D} ) \star \mathcal{C} )
= \mathrm{det} (\lambda \mathcal{I}_{mmp} - \mathcal{D} \star \mathcal{C} ).
\end{eqnarray}
This shows that $\sigma(\mathcal{C} \star \mathcal{D}) = \sigma(\mathcal{D} \star \mathcal{C})$. 

If $\mathcal{C}$ is not invertible, we choose a sequence $\{\epsilon_n \}$ in $\mathbb{C} \backslash \sigma(\mathcal{C})$ with $\epsilon_n \rightarrow 0$, with property that all new tensors $\mathcal{C}_n \define \mathcal{C} - \epsilon_n \mathcal{I}_{mmp}$ are invertibale for each $n$. Then, 
\begin{eqnarray}
\mathrm{det} ( \lambda \mathcal{I}_{mmp} - \mathcal{C} \star \mathcal{D} ) & = &
\lim\limits_{n \rightarrow \infty} \mathrm{det} (  \lambda \mathcal{I}_{mmp} - \mathcal{C}_n \star \mathcal{D} ) = \lim\limits_{n \rightarrow \infty} \mathrm{det} (  \lambda \mathcal{I}_{mmp} -  \mathcal{D}  \star \mathcal{C}_n ) \nonumber \\\
& = & \mathrm{det} (  \lambda \mathcal{I}_{mmp} - \mathcal{D} \star \mathcal{C} ). 
\end{eqnarray}
$\hfill \Box$

\begin{lemma}\label{lma:2.5.1 Hiai paper 2010}
For every tensor $\mathcal{C} \in  \mathbb{C}^{m \times m \times p}$ and every function $f$ on $\sigma(\mathcal{C}^{\mathrm{H}}  \star \mathcal{C} )$, we have 
\begin{eqnarray}
\mathcal{C} \star f (  \mathcal{C}^{\mathrm{H}}  \star \mathcal{C} )
&=& f (  \mathcal{C}  \star \mathcal{C}^{\mathrm{H}} ) \star \mathcal{C}. 
\end{eqnarray}
\end{lemma}
\textbf{Proof:}
Since $\sigma( \mathcal{C}^{\mathrm{H}} \star \mathcal{C}) = \sigma(\mathcal{C} \star \mathcal{C}^{\mathrm{H}}  )$ from Lemma~\ref{lma:1.5.7 Hiai paper 2010} and $\mathcal{C} \star \left( \mathcal{C}^{\mathrm{H}}  \star \mathcal{C} \right)^n = \left( \mathcal{C} \star \mathcal{C}^{\mathrm{H}}  \right)^n \star \mathcal{C}  $, for $n \in \mathbb{N}$, this lemma is hold for $f$ is a polynomial. If the function $f$ is an arbitary function on $\sigma(  \mathcal{C}^{\mathrm{H}} \star \mathcal{C} )= [s_1, \cdots s_n]$, we define the Lagrance interpolation polynomial as 
\begin{eqnarray}
p(x) \define \sum\limits_{i=1}^{n} f(s_i) \prod\limits_{1 \leq j \leq n, i \neq j} \frac{x - s_j}{s_i - s_j,}
\end{eqnarray}
where we have $p(s_i) = f(s_i)$ for $1 \leq i \leq n$. Then, we also have 
\begin{eqnarray}
\mathcal{C} \star f (  \mathcal{C}^{\mathrm{H}}  \star \mathcal{C} )
=  \mathcal{C} \star p (  \mathcal{C}^{\mathrm{H}}  \star \mathcal{C} ) =  p (  \mathcal{C} \star \mathcal{C}^{\mathrm{H}}  ) \star \mathcal{C} = f (  \mathcal{C}  \star \mathcal{C}^{\mathrm{H}} ) \star \mathcal{C},
\end{eqnarray}
and this  Lemma is proved. $\hfill \Box$

Following Lemma is adopted from Corollary 12 from~\cite{chansangiam2015survey}. 
\begin{lemma}\label{lma:cor 12 from operator survey}
We have following equivalent statements about a function $f(x): (0, \infty) \rightarrow (0, \infty)$:
\begin{enumerate}
\item $f(x)$ is T-product tensor monotone function; 
\item $x/f(x)$ is T-product tensor monotone function;
\item $f(x)$ is T-product tensor concave function;
\item $1/f(x)$ is T-product tensor convex function. 
\end{enumerate}
\end{lemma}
\textbf{Proof:}
The proof is based in Corollary 12 from~\cite{chansangiam2015survey}. But those facts about using Theorem 2.5.2 and Theorem 2.5.3 from~\cite{hiai2010matrix} should be modified from matrices settings to T-product tensors settings. With help from Lemma~\ref{lma:2.5.1 Hiai paper 2010} and transfer rules provided by Eq.~\eqref{eq:transfer rule},  the proof about Theorem 2.5.2 and Theorem 2.5.3 from~\cite{hiai2010matrix} for T-product tensors is straightforward by replacing matrix multiplication opertions to T-product operations.  $\hfill \Box$

We are ready to prove that the logarithmic function is T-product tensor monotone and
concave function on $(0, \infty)$. 
\begin{lemma}\label{lma:log monotone and concave}
Given two TPD tensors $\mathcal{C}, \mathcal{D} \in \mathbb{C}^{m \times m \times p}$ with $\mathcal{O} \preceq \mathcal{C} \preceq \mathcal{D} $, we have 
\begin{eqnarray}
\log ( \mathcal{C}  ) \preceq \log ( \mathcal{D}  ),
\end{eqnarray}
and
\begin{eqnarray}
t  \log ( \mathcal{C}  ) + (1 - t) \log ( \mathcal{D}  ) \preceq \log (t \mathcal{C} + (1-t) \mathcal{D} ).
\end{eqnarray}
\end{lemma}
\textbf{Proof:}
We define a function $g(x) = \frac{x }{\log (x + 1) }$ on $(0, \infty)$. Since $g(x)$ is the monotone function on $(0, \infty)$, Lemma~\ref{lma:cor 12 from operator survey} implies that $\log(1  + x)$ is T-product tensor monotone and concave on $(0, \infty)$. For each $\epsilon > 0$, $\log (\epsilon + x) = \log \epsilon + \log (1 + x/\epsilon)$ is T-product tensor  monotone and concave on $(0, \infty)$. Let $\epsilon \rightarrow 0$, we achive the desired result. $\hfill \Box$

In general, it is not practical to always working with Hermitian T-product tensor, we will apply dilations techqnique to expand any Ttensor into a Hermitian T-product tensor. For any tensor $\mathcal{C} \in \mathbb{C}^{m \times n \times p}$, a dilation for the tensor  $\mathcal{C}$, denoted as $\mathfrak{D}(\mathcal{C})$, will be
\begin{eqnarray}\label{eq:dilation T tensor}
\mathfrak{D}(\mathcal{C} )\define  \left[
    \begin{array}{cc}
     \mathcal{O} & \mathcal{C} \\
     \mathcal{C}^{\mathrm{H}} & \mathcal{O} \\
    \end{array}
\right],
\end{eqnarray}
where $\mathfrak{D}(\mathcal{C} ) \in  \mathbb{C}^{(m+n) \times (m+n) \times p}$ and we have $ \left( \mathfrak{D}(\mathcal{C} ) \right)^{\mathrm{H}} =   \mathfrak{D}(\mathcal{C} ) $ (Hermitian T-product tensor after dilation). Also, we have 
\begin{eqnarray}\label{eq:dilation T tensor 2}
\mathfrak{D}^2(\mathcal{C} )\define  \left[
    \begin{array}{cc}
     \mathcal{C} \star \mathcal{C}^{\mathrm{H}}  & \mathcal{O} \\
     \mathcal{O} &  \mathcal{C}^{\mathrm{H}} \star \mathcal{C} \\
    \end{array}
\right].
\end{eqnarray}
Since Eq.~\eqref{eq:dilation T tensor} is zero trace, the largest eigenvalue of $\mathfrak{D}(\mathcal{C} )$ will be the same with the largest singular of $\mathcal{C}$.

Since the expectation of a random T-product tensor can be considered as a convex combination, expectation preserves the TPSD order as:
\begin{eqnarray}\label{eq: exp keep TPSD order}
\mathcal{X} \preceq \mathcal{Y} ~~ \mbox{almost surely}~~ \Longrightarrow  \mathbb{E} \mathcal{X} \preceq  \mathbb{E} \mathcal{Y}. 
\end{eqnarray}
Also from Lemma~\ref{lma:cor 12 from operator survey}, we know that the quadratic function $f(x) = x^2$ is T-product tensor convex, thus, we have 
\begin{eqnarray}\label{eq: quadratic Jensen TPSD order}
\left( \mathbb{E} \mathcal{C} \right)^2 \preceq \mathbb{E} \left( \mathcal{C}^2 \right).
\end{eqnarray}

We will present one more theorem in this section about Golden-Thompson inequality for T-product tensors.  We recall~\cref{thm:GT Inequality for T-product Tensors intro}

\ThmGTInequalityTProduct*
%\begin{theorem}[ Golden-Thompson inequality for T-product Tensors]\label{thm:GT Inequality for T-product Tensors}
%%
%Given two Hermitian T-product tensors $\mathcal{C}, \mathcal{D} \in \mathbb{C}^{m \ times m \times p}$, we have 
%%
%\begin{eqnarray}
%\mathrm{Tr} \left( \exp(\mathcal{C} + \mathcal{D}) \right) \leq \mathrm{Tr} \left( \exp \left( \mathcal{C} \right)  \star \exp \left( \mathcal{D} \right) \right)
%\end{eqnarray}
%%
%\end{theorem}
%%
\textbf{Proof:}
From T-SVD decomposition and Eq.~\eqref{eq:spectral expression}, we can express the tensor $\mathcal{C}$ as 
\begin{eqnarray}
\mathcal{C} = \sum\limits_{\lambda} \mathcal{P}_{\lambda},
\end{eqnarray}
where $\lambda$ are eigenvalues and $\mathcal{P}_{\lambda}$ are corresponding projectors (T-product tensors) which are mutually orthgonal. Given $\mathcal{X} \succeq \mathcal{O}$, we define following mapping with respect to the tensor $\mathcal{C}$ as :
\begin{eqnarray}\label{eq:pinching map}
\mathfrak{P}_{\mathcal{C}}(\mathcal{X}) : \mathcal{X} \rightarrow \sum\limits_{\lambda}
\mathcal{P}_{\lambda} \star \mathcal{X} \star \mathcal{P}_{\lambda}.
\end{eqnarray}
Then, we have following properties about mapping $\mathfrak{P}_{\mathcal{C}}(\mathcal{X})$
\begin{enumerate}\label{pinching properties}
\item  $\mathfrak{P}_{\mathcal{C}}(\mathcal{X})$ commutes with $\mathcal{C}$;
\item  $\mathrm{Tr}\left( \mathfrak{P}_{\mathcal{C}}(\mathcal{X}) \star \mathcal{C} \right) = \mathrm{Tr} \left( \mathcal{X} \star \mathcal{C} \right)$;
\item  $\mathfrak{P}_{\mathcal{C}}(\mathcal{X}) \succeq \frac{\mathcal{X}}{\left\vert sp(\mathcal{C}) \right\vert } $, where $sp(\mathcal{C}) = \left\{ \lambda_1, \lambda_2, \cdots, \lambda_{ \left\vert sp(\mathcal{C}) \right\vert}\right\}$. 
\end{enumerate}
The third property of $\mathfrak{P}_{\mathcal{C}}(\mathcal{X})$ is true due to the following relation:
\begin{eqnarray}
\mathfrak{P}_{\mathcal{C}}(\mathcal{X}) &=& \sum\limits_{\lambda \in sp(\mathcal{C}) }  \mathcal{P}_{\lambda} \star \mathcal{X} \star \mathcal{P}_{\lambda} \nonumber \\
&=& \frac{1}{  \left\vert sp(\mathcal{C}) \right\vert   } \sum\limits_{x=1}^{  \left\vert sp(\mathcal{C}) \right\vert } \mathcal{U}_x \star \mathcal{X} \star \mathcal{U}^{\mathrm{H}}_x \nonumber \\
& \succeq &  \frac{\mathcal{X}}{ \left\vert sp(\mathcal{C}) \right\vert },
\end{eqnarray}
where $\mathcal{U}_x = \sum\limits_{i=1}^{  \left\vert sp(\mathcal{C}) \right\vert          } \exp\left( \frac{ \sqrt{-1} 2 \pi x i  }{  \left\vert sp(\mathcal{C}) \right\vert  } \right) \mathcal{P}_{\lambda_i} $. 

Let $\mathcal{A}_1 = \exp(\mathcal{C})$ and $\mathcal{A}_2 = \exp(\mathcal{D})$, we have 
\begin{eqnarray}
\log \mathrm{Tr}\left(\exp \left( \log \mathcal{A}_1 + \log \mathcal{A}_2 \right) \right) &=_1& \frac{1}{n} \log \mathrm{Tr} \left( \exp \left( \log \mathcal{A}^{\otimes n}_1  + \log \mathcal{A}^{\otimes n}_2\right)\right) \nonumber \\
&\leq_2&  \frac{1}{n} \log \mathrm{Tr} \left( \exp \left( \log \mathfrak{P}_{\mathcal{A}^{\otimes n}_2}( \mathcal{A}_1^{\otimes n})  + \log \mathcal{A}^{\otimes n}_2\right)\right) \nonumber \\
& &+ \frac{\log \mbox{poly} (n) }{n} \nonumber \\
& =_3 & \frac{1}{n}\log \left( \mathrm{Tr}\left(  \mathfrak{P}_{\mathcal{A}^{\otimes n}_2}( \mathcal{A}_1^{\otimes n}) \star \mathcal{A}^{\otimes n}_2     \right)      \right) + \frac{\log \mbox{poly} (n) }{n} \nonumber \\
& =_4 & \log\mathrm{Tr} \left( \mathcal{A}_1 \star \mathcal{A}_2 \right) +  \frac{\log \mbox{poly} (n) }{n}.
\end{eqnarray}
he equality $=_1$ comes from  the fact that the trace is multiplicative under the Kronecker product. The inequality $\leq_2$ follows from inequality from the third property of $\mathfrak{P}_{\mathcal{A}^{\otimes m}_2}( \mathcal{A}_1^{\otimes m})$, the monotone of $\log$ and $\mathrm{Tr} \exp \left( ~\right)$ functions, and the number of eigenvalues of $\mathcal{A}_2^{\otimes n}$ growing polynomially with $n$ due to the fact that the number of distinct eigenvalues of $\mathcal{A}_2^{\otimes n}$ is bounded by the number of different types of sequences of $mp$ symbols of length $n$, see Lemma II.1 in~\cite{csiszar1998method}. The equality $=_3$ utilizes the commutativity property for tensors $\mathfrak{P}_{\mathcal{A}^{\otimes m}_2}( \mathcal{A}_1^{\otimes m}) $ and $\mathcal{A}^{\otimes m}_2$ based on the first property. Finally, the equality $=_4$ applies trace properties from the second property of the mapping $\mathfrak{P}_{\mathcal{A}^{\otimes m}_2}( \mathcal{A}_1^{\otimes m})$.  If $n \rightarrow \infty$,  the result of this theorem is established.
$\hfill \Box$

\section{Lieb's Concavity Under T-product}\label{sec:Lieb's Concavity Under Tproduct}

In this section, we will extend several trace inequalities to T-product tensors: Jensen's T-product tensor inequality in Section~\ref{subsec:Jesens Tproduct Inequality} and Klein's T-product tensor inequality in Section~\ref{subsec:Klein's T-product Inequality}. These new T-product tensor inequalities will play important roles in establishing a new version of Lieb's concavity theorem under T-product in Section~\ref{subsec:Liebs Concavity Theorem Under Tproduct}.

\subsection{Jensen's T-product Inequality}\label{subsec:Jesens Tproduct Inequality}

In this subsection, we will derive Jensen's T-product tensor inequality in Theorem~\ref{thm:Jensen's T-product Inequality intro}. We begin with a lemma which will be used in later proof in Thereom~\ref{thm:Jensen's T-product Inequality intro}. 

Given two natural numbers $m, n$, we define a T-product tensor $\theta \in \mathbb{C}^{m \times m \times p}$ as $\exp ( 2 \pi \sqrt{-1}/ n) \times \mathcal{I}_{mmp}$. Then, we can have tensor $\mathcal{D} \in \mathbb{C}^{mn \times mn \times p}$ obtained by 
\begin{eqnarray}\label{eq:diagonal unitary tensor 14}
\mathcal{D} = \mbox{diag} \left(\overbrace{\theta, \theta^2, \cdots, \theta^{n-1}, \mathcal{I
}_{mmp}}^{\mbox{total $n$ T-product tensors}} \right),
\end{eqnarray}
where $ \mbox{diag} \left(\theta, \theta^2, \cdots, \theta^{n-1}, \mathcal{I
}_{mmp} \right)$ will be a matrix with entries as T-product tensors, and the \emph{diagonal} part of this matrix is compsoded by tensors  $(\theta, \theta^2, \cdots, \theta^{n-1}, \mathcal{I}_{mmp})$. Let $\mathcal{D} \in \mathbb{C}^{mn \times mn \times p}$ be another matrix of T-product tensor, i.e., entries $d_{i, j}$ as T-product tensors. We define a new operation $\oast$ between two T-product tensors, $\mathcal{A}, \mathcal{B}$ with dimensions belong to $\mathbb{C}^{mn \times mn \times p}$ as:
\begin{eqnarray}
\left( \mathcal{A} \oast \mathcal{B} \right) \define \sum\limits_{k=1}^n a_{i, k} \star b_{k, j}, 
\end{eqnarray}
where both $a_{i, k}$ and $c_{k, j}$ are T-product tensors. Therefore, given any tensor $\mathcal{C} \in \mathbb{C}^{mn \times mn \times p}$, the $i, j$-th entry (a T-product tensor) of $\mathcal{C} \oast \mathcal{D}$ becoms $ \exp ( 2 \pi \sqrt{-1} j / n) \times c_{i, j}$, where $c_{i, j} \in \mathbb{C}^{m \times m \times p}$ is a T-product tensor. 

\begin{lemma}\label{lma:3.1 Jensen Operator Inequality Frank}
Given any tensor $\mathcal{C} \in \mathbb{C}^{mn \times mn \times p}$ and the tensor $\mathcal{D}$ defined by Eq.~\ref{eq:diagonal unitary tensor 14}, we have 
\begin{eqnarray}
\frac{1}{n} \sum\limits_{k=1}^{n} \mathcal{D}^{-k} \oast \mathcal{C} \oast \mathcal{D}^k
&=& \mbox{diag} \left(c_{1, 1}, c_{2, 2}, \cdots, c_{n, n}\right)
\end{eqnarray}
where $\mathcal{D}^k$ is the self-product of the tensor $\mathcal{D}$ by $\oast$ operation $k$-times. 
\end{lemma}
\textbf{Proof:}
By direct computation with $\oast$, we have following:
\begin{eqnarray}
\left(\frac{1}{n} \sum\limits_{k=1}^{n} \mathcal{D}^{-k} \oast \mathcal{C} \oast \mathcal{D}^k \right)_{i, j} = \frac{1}{n}  \sum\limits_{k=1}^{n} \left(  \exp ( 2 \pi \sqrt{-1} (j - i) / n)  \right)^k c_{i, j}, 
\end{eqnarray}
where this summation is zero for $ i \neq j$, otherwise, it is $c_{i, i}$. $\hfill \Box$

We are ready to prove~\cref{thm:Jensen's T-product Inequality intro}. 

\ThmJensenTProductInequality*

%\begin{theorem}\label{thm:Jensen's T-product Inequality}
%%
%For a continous \emph{T-product tensor convex} function $f$ defined on an interval $\mathrm{I}$~\footnote{The definition of \emph{T-product tensor convex} is given by Eq.~\eqref{eq:T tensor convex}}, we have following TPSD relation for each natural number $n$:
%%
%\begin{eqnarray}\label{eq1:thm:Jensen's T-product Inequality}
%f \left( \sum\limits_{i=1}^{n} \mathcal{C}_i^{\mathrm{H}} \star \mathcal{X}_i \star \mathcal{C}_i \right) \preceq 
%\sum\limits_{i=1}^{n} \mathcal{C}_i^{\mathrm{H}} \star f\left(   \mathcal{X}_i \right)\star  \mathcal{C}_i,
%\end{eqnarray}
%%
%where $\mathcal{X}_i \in \mathbb{C}^{mmp}$ are bounded, Hermitian T-product tensor with all eigenvalues in $\mathrm{I}$ and tensors $\mathcal{C}$ satisfying $\sum\limits_{i=1}^{n} \mathcal{C}_i^{\mathrm{H}} \star  \mathcal{C}_i  = \mathcal{I}_{mmp}$. 
%%
%\end{theorem}
% \oast
\textbf{Proof:}
Let us define a unitary tensor $\mathcal{U} = (u_{i, j}) \in \mathbb{C}^{mn \times mn \times p}$ for $1 \leq i, j \leq n$ as $u_{i, j} = \mathcal{C}_i$,  $\mathcal{D} = \mbox{diag} \left( \theta, \theta^2, \cdots, \theta^{n-1}, \mathcal{I}_{mmp} \right)$ defined by Eq.~\eqref{eq:diagonal unitary tensor 14}, and define the tensor $\overline{\mathcal{X}} \in \mathbb{C}^{mn \times mn \times p}$ as $ \mbox{diag} \left(\mathcal{X}_1, \cdots, \mathcal{X}_n \right)$. From Lemma~\ref{lma:3.1 Jensen Operator Inequality Frank}, we have 
\begin{eqnarray}\label{eq1:thm:Jensen's T-product Inequality}
f \left( \sum\limits_{i=1}^{n} \mathcal{C}_i^{\mathrm{H}} \star \mathcal{X}_i \star \mathcal{C}_i \right) &=& f \left(    \left(   \mathcal{U}^{\mathrm{H}} \oast \overline{\mathcal{X}}   \oast \mathcal{U}        \right)_{n, n}  \right)  \nonumber \\
&=& f \left( \left( \sum\limits_{i=1}^{n} \frac{1}{n} \mathcal{D}^{-i} \oast \mathcal{U}^{\mathrm{H}} \oast \overline{\mathcal{X}} \oast   \mathcal{U} \oast \mathcal{D}^{i}  \right)_{n, n}\right)  \nonumber \\
&=&  f \left( \left( \sum\limits_{i=1}^{n} \frac{1}{n} \mathcal{D}^{-i} \oast \mathcal{U}^{\mathrm{H}} \oast \overline{\mathcal{X}} \oast   \mathcal{U} \oast \mathcal{D}^{i}  \right)\right)_{n, n}  \nonumber \\
& \leq & \left( \frac{1}{n}  \sum\limits_{i=1}^{n}  f \left( \mathcal{D}^{-i} \oast \mathcal{U}^{\mathrm{H}} \oast \overline{\mathcal{X}} \oast   \mathcal{U} \oast \mathcal{D}^{i}         \right) \right)_{n, n} \nonumber \\
&=&  \left( \frac{1}{n}  \sum\limits_{i=1}^{n}  \mathcal{D}^{-i} \oast \mathcal{U}^{\mathrm{H}} \oast f \left( \overline{\mathcal{X}} \right) \oast   \mathcal{U} \oast \mathcal{D}^{i}         \right)_{n, n} \nonumber \\
&=& \left(  \mathcal{U}^{\mathrm{H}} \oast f \left( \overline{\mathcal{X}} \right) \oast   \mathcal{U} \right)_{n, n}  \nonumber \\
&=& \sum\limits_{i=1}^{n} \mathcal{C}_i^{\mathrm{H}} \star f( \mathcal{X}_i )  \star \mathcal{C}_i,  
\end{eqnarray}
where the inequality comes from that the function $f$ is a \emph{T-product tensor convex} function. $\hfill \Box$

\subsection{Klein's T-product Inequality}\label{subsec:Klein's T-product Inequality}

The immediate application of Theorem~\ref{thm:convexity and monotonicity of a trace func intro} is to prove Klein's inequality for T-product tensor. We recall~\cref{thm:Klein's Ineqaulity T-product intro}.

\ThmKleinsTProductInequality*

%
%%
%\begin{theorem}\label{thm:Klein's Ineqaulity T-product}
%%
%For all $\mathcal{C}, \mathcal{D}$ Hermitian T-product tensors and a differentiable convex function $f: \mathbb{R} \rightarrow \mathbb{R}$ or for all $\mathcal{C}, \mathcal{D}$ Hermitian T-product tensors and a differentiable convex function $f: (0, \infty) \rightarrow \mathbb{R}$, we have
%%
%\begin{eqnarray}
%\mathrm{Tr}\left( f(\mathcal{C}) -  f(\mathcal{D}) - (\mathcal{C} - \mathcal{D}) \star f'(\mathcal{D}) \right) \geq 0. 
%\end{eqnarray}
%%
%In both situations, if $f$ is strictly convex, equality holds if and only if $\mathcal{C} = \mathcal{D}$. 
%%
%\end{theorem}
%%
\textbf{Proof:}
We define function $F(t)$ as 
\begin{eqnarray}\label{eq1:thm:Klein's Ineqaulity T-product}
F(t) = \mathrm{Tr}\left( f \left( \mathcal{D} + t \left(  \mathcal{C} -  \mathcal{D} \right)  \right)  \right),
\end{eqnarray}
where $t \in (0, 1)$. From Theorem~\ref{thm:convexity and monotonicity of a trace func intro}, $F(t)$ is a convex function. Then, we have
\begin{eqnarray}
F(0) + t ( F(1) - F(0) ) \geq F(t) \Longleftrightarrow F(1) - F(0)  \geq \frac{F(t) -F(0)}{t}
\end{eqnarray}
By taking limit $t \rightarrow 0$ at $F(1) - F(0)  \geq \frac{F(t) -F(0)}{t}$, we have 
\begin{eqnarray}
F(1) - F(0)  \geq F'(0),
\end{eqnarray}
then we obatin Klein's ineqaulity under T-product by
rearrangement and substitution with Eq.~\eqref{eq1:thm:Klein's Ineqaulity T-product}. $\hfill \Box$

\subsection{Lieb's Concavity Theorem Under T-product}\label{subsec:Liebs Concavity Theorem Under Tproduct}

In this section, we will extend Lieb’s concavity theorem to T-product tensors and we begin with the definition about the relative entropy between two T-product tensors. 

\begin{definition}\label{def: relative entropy for tensors}
Given two TPD tensors $\mathcal{A} \in \mathbb{C}^{m \times m \times p}$ and tensor $\mathcal{B} \in \mathbb{C}^{m \times m \times p}$. The relative entropy between two T-product tensors $\mathcal{A}$ and  $\mathcal{B}$ is defined as 
\begin{eqnarray}\label{eq:relative entropy T-product}
D( \mathcal{A} \parallel \mathcal{B}) \define \mathrm{Tr} \mathcal{A} \star (\log \mathcal{A} - \log \mathcal{B}).
\end{eqnarray} 
\end{definition}

We apply \emph{perspective function} concept for T-product tensor convex and introduce the following lemma about the convexity of a T-product tensor convex function~\cite{MR2475796}. 
\begin{lemma}\label{lma:perspective func}
Given $f$ as a convex function, two commuting tensors $\mathcal{X}, \mathcal{Y} \in \mathbb{C}^{m \times m \times p}$, i.e., $\mathcal{X} \star \mathcal{Y} = \mathcal{Y}\star \mathcal{X}$, and the existence of the $\mathcal{Y}^{-1}$, then the following map $h$ 
\begin{eqnarray}\label{eq:lma:perspective func}
h (\mathcal{X}, \mathcal{Y}) = f(\mathcal{X} \star \mathcal{Y}^{-1}) \star \mathcal{Y}
\end{eqnarray}
is jointly convex in the sense that, given $t \in [0, 1]$, if $\mathcal{X} = t \mathcal{X}_1 + (1-t) \mathcal{X}_2$ and $\mathcal{Y} = t \mathcal{Y}_1 + (1-t) \mathcal{Y}_2$ with $\mathcal{X}_1 \star \mathcal{Y}_1 = \mathcal{Y}_1 \star \mathcal{X}_1$ and $\mathcal{X}_2 \star \mathcal{Y}_2 = \mathcal{Y}_2 \star \mathcal{X}_2$, we should have 
\begin{eqnarray}
h (\mathcal{X}, \mathcal{Y})  \leq t h (\mathcal{X}_1, \mathcal{Y}_1) + (1-t) h (\mathcal{X}_2, \mathcal{Y}_2). 
\end{eqnarray}
\end{lemma}
\begin{proof}
Constructing tensors $\mathcal{A} = (t \mathcal{Y}_1)^{1/2} \star \mathcal{Y}^{-1/2}$
and $\mathcal{B} = ( (1-t)\mathcal{Y}_2)^{1/2} \star \mathcal{Y}^{-1/2}$, then we have 
\begin{eqnarray}\label{eq:lma:perspective func 1}
\mathcal{A}^{\mathrm{H}} \star \mathcal{A} + \mathcal{B}^{\mathrm{H}} \star \mathcal{B} = \mathcal{I}_{mmp}
\end{eqnarray}

Since we have 
\begin{eqnarray}
h(\mathcal{X}, \mathcal{Y}) &=& f(\mathcal{X} \star \mathcal{Y}^{-1} ) \star \mathcal{Y} \nonumber \\
&=&   \mathcal{Y}^{1/2} \star f(  \mathcal{Y}^{-1/2} \star \mathcal{X} \star \mathcal{Y}^{-1/2} ) \star \mathcal{Y}^{1/2} \nonumber \\
&=&  \mathcal{Y}^{1/2} \star f( \mathcal{A}^{\mathrm{H}} \star \mathcal{X}_1 \star \mathcal{Y}^{-1}_1 \star \mathcal{A} +  \mathcal{B}^{\mathrm{H}} \star  \mathcal{X}_2 \star \mathcal{Y}^{-1}_2  \star \mathcal{B} ) \star \mathcal{Y}^{1/2} \nonumber \\
&\leq_1& \mathcal{Y}^{1/2} \star \left( \mathcal{A}^{\mathrm{H}} \star f(\mathcal{X}_1 \star \mathcal{Y}^{-1}_1 ) \star \mathcal{A} \right. \nonumber \\
&  & \left. +  \mathcal{B}^{\mathrm{H}} \star f(\mathcal{X}_2 \star \mathcal{Y}^{-1}_2 )  \star \mathcal{B} \right) \star \mathcal{Y}^{1/2} \nonumber \\
&=&  (t \mathcal{Y}_1)^{1/2}  f(\mathcal{X}_1 \star \mathcal{Y}^{-1}_1 )  (t \mathcal{Y}_1)^{1/2} + ((1-t) \mathcal{Y}_2)^{1/2}  f(\mathcal{X}_2 \star \mathcal{Y}^{-1}_2 )  ((1-t) \mathcal{Y}_2)^{1/2}   \nonumber \\
&=& t h(\mathcal{X}_1, \mathcal{Y}_1)  + (1 - t) h(\mathcal{X}_2, \mathcal{Y}_2) 
\end{eqnarray}
where $\leq_1$ is based on the condition provided by Eq.~\eqref{eq:lma:perspective func 1} and Theorem~\ref{thm:Jensen's T-product Inequality intro}. 
\end{proof}

Following lemma is given to establish the joint convexity property of relative entropy for T-product tensors.
\begin{lemma}[Joint Convexity of Relative Entropy for T-product Tensors]\label{lma:joint conv of rela entropy}
The relative entropy function of two TPD tensors  is a jointly convex function. That is 
\begin{eqnarray}\label{eq:lma:joint conv of rela entropy}
\mathbb{D}(t \mathcal{A}_1 + (1 - t) \mathcal{A}_2 \parallel t \mathcal{B}_1 + (1 - t) \mathcal{B}_2) \leq t \mathbb{D}( \mathcal{A}_1 \parallel \mathcal{B}_1 ) +  (1- t) \mathbb{D} (\mathcal{A}_2 \parallel \mathcal{B}_2),
\end{eqnarray} 
where $ t \in [0, 1]$ and all the following four tensors $\mathcal{A}_1$, $\mathcal{B}_1$, $\mathcal{A}_2$ and $\mathcal{B}_2$, are TPD tensors. 
\end{lemma}
\textbf{Proof:}
From the definition~\ref{def: relative entropy for tensors}, we wish to show the joint convexity of the function $\mathbb{D} (\mathcal{A} \parallel \mathcal{B})$ with respect to the tensors $\mathcal{A}, \mathcal{B} \in \mathbb{C}^{m \times m \times p}$. Let us define tensor operators $\mathcal{F}(\mathcal{X}) \define \mathcal{A} \star \mathcal{X}$ and $\mathcal{G}(\mathcal{X}) \define \mathcal{X} \star \mathcal{B}$  for the variable tensor $\mathcal{X}  \in \mathbb{C}^{m \times m \times p}$. Then, we have $\mathcal{F}(\mathcal{X})$ and $\mathcal{G}(\mathcal{X})$ commuting on the inner product operation $\langle \mathcal{F}(\mathcal{X}), \mathcal{G}(\mathcal{X}) \rangle$ defined as:
\begin{eqnarray}
\langle \mathcal{F}(\mathcal{X}), \mathcal{G}(\mathcal{X}) \rangle &=& 
\mathrm{Tr}(\mathcal{F}^{\mathrm{H}}(\mathcal{X})\star \mathcal{G}(\mathcal{X}) )
\end{eqnarray}
Then, we have $\mathrm{Tr}(\mathcal{F}^{\mathrm{H}}(\mathcal{X})\star \mathcal{G}(\mathcal{X}) ) =\mathrm{Tr}(\mathcal{G}^{\mathrm{H}}(\mathcal{X}) \star \mathcal{F}(\mathcal{X}) )$. Since the function $f(x) = x \log x$ is tensor convex, we apply Lemma~\ref{lma:perspective func} to operators $\mathcal{F}(~~), \mathcal{G}(~~)$ and the function $h$ definition provided by Eq.~\eqref{eq:lma:perspective func} to obtain the following relation ($\mathcal{I} = \mathcal{I}_{mmp}$ in this proof): 
\begin{eqnarray}
\langle \mathcal{I}, h (\mathcal{F}(\mathcal{I}), \mathcal{G}(\mathcal{I})) \rangle &=& 
\langle \mathcal{I},~~\mathcal{G}(\mathcal{I}) \star ( \mathcal{F}(\mathcal{I}) \star \mathcal{G}^{-1}(\mathcal{I} )) \log ( \mathcal{F}(\mathcal{I} ) \star \mathcal{G}^{-1}(\mathcal{I})  )   \rangle \nonumber \\
&=& \langle \mathcal{I}, \mathcal{F}(\mathcal{I}) (\log \mathcal{F}(\mathcal{I}) - \log \mathcal{G}(\mathcal{I}) )   \rangle \nonumber \\
&=& \mathrm{Tr}(\mathcal{A} \log \mathcal{A} - \mathcal{A} \log \mathcal{B} )= \mathbb{D}(\mathcal{A} \parallel \mathcal{B}),
\end{eqnarray}
is jointly convex with respect to tensors $\mathcal{A}$ and $\mathcal{B}$. 
$\hfill \Box$

Lieb's concavity theorem is recalled below by~\cref{thm:Lieb concavity thm}.

\ThmLiebConcavityTProduct*

%\begin{theorem}[Lieb's concavity theorem for T-product tensors]\label{thm:Lieb concavity thm}
%%
%Let $\mathcal{H}$ be a Hermitian T-product tensor. Following map 
%%
%\begin{eqnarray}\label{eq:thm:Lieb concavity thm}
%\mathcal{A} \rightarrow \mathrm{Tr}e^{\mathcal{H}+ \log \mathcal{A}}
%\end{eqnarray}
%%
%is concave on the positive-definite cone.
%%
%\end{theorem}
%%
\textbf{Proof:}
From Klein's inequality obtain from Theorem~\ref{thm:Klein's Ineqaulity T-product intro}, the convexity of map $t \rightarrow t \log t$ (which is strictly concave for $t > 0$) and Hermitian T- tensors $\mathcal{X}, \mathcal{Y}$, we have 
\begin{eqnarray}
\mathrm{Tr} \mathcal{Y} \geq \mathrm{Tr} \mathcal{X} - \mathrm{Tr} \mathcal{X} \log \mathcal{X} + \mathrm{Tr} \mathcal{X} \log \mathcal{Y}.
\end{eqnarray}
If we replace $\mathcal{Y}$ by $e^{\mathcal{H} + \log \mathcal{A}}$, we then have
\begin{eqnarray}\label{eq:var formula}
\mathrm{Tr} e^{\mathcal{H} + \log \mathcal{A}} = \max\limits_{\mathcal{X} \succ \mathcal{O}} \Big\{ \mathrm{Tr}\mathcal{X} \star \mathcal{H} - \mathbb{D}(\mathcal{X} \parallel \mathcal{A})  + \mathrm{Tr}\mathcal{X} \Big\}
\end{eqnarray}
where $\mathbb{D}(\mathcal{X} \parallel  \mathcal{A})$ is the quantum relative entropy between two tensor operators. For real number $t \in [0, 1]$ and two positive-definite tensors $\mathcal{A}_1, \mathcal{A}_2$, we have 
\begin{eqnarray}
\mathrm{Tr} e^{\mathcal{H} + \log (t \mathcal{A}_1 + (1-t) \mathcal{A}_2 )} &=& 
\max_{\mathcal{X} \succ \mathcal{O}} \Big\{ \mathrm{Tr} \mathcal{X}\mathcal{H} 
- \mathbb{D}(\mathcal{X} \parallel t \mathcal{A}_1 + (1-t) \mathcal{A}_2) + \mathrm{Tr}\mathcal{X} \Big\}  \nonumber \\
&\geq & t \max_{\mathcal{X} \succ \mathcal{O}} \Big\{ \mathrm{Tr} \mathcal{X}\mathcal{H} 
- \mathbb{D}(\mathcal{X} \parallel  t \mathcal{A}_1) 
+ \mathrm{Tr}\mathcal{X} \Big\} \nonumber \\ 
& &+ (1-t) \max_{\mathcal{X} \succ \mathcal{O}} \Big\{ \mathrm{Tr} \mathcal{X}\mathcal{H} 
- \mathbb{D}(\mathcal{X}  \parallel  (1-t) \mathcal{A}_2) + \mathrm{Tr}\mathcal{X} \Big\} \nonumber \\
& = & t \mathrm{Tr}e^{\mathcal{H} + \log \mathcal{A}_1} +  (1-t) \mathrm{Tr}e^{\mathcal{H} + \log \mathcal{A}_2},
\end{eqnarray}
where the first and last equalities are obtained based on the variational formula provided by Eq.~\eqref{eq:var formula}, and the inequality is due to the joint convexity property of the 
relative entropy from Leamm~\ref{lma:joint conv of rela entropy}.
$\hfill \Box$

Based on Lieb's concavity theorem for T-product tensors, we have the following corollary.
\begin{corollary}\label{cor:3.3}
Let $\mathcal{A}$ be a fixed Hermitian T-product tensor, and let $\mathcal{X}$ be a random Hermitian T-product tensor, then we have
\begin{eqnarray}
\mathbb{E} \mathrm{Tr} e^{\mathcal{A} + \mathcal{X}} \leq \mathrm{Tr} e^{\mathcal{A} + \log \left( \mathbb{E} e^{\mathcal{X}} \right) }.
\end{eqnarray}
\end{corollary}
\begin{proof}
Define the random tensor $\mathcal{Y} = e^{\mathcal{X}}$, we have
\begin{eqnarray}
\mathbb{E} \mathrm{Tr} e^{\mathcal{A} + \mathcal{X}} = 
\mathbb{E} \mathrm{Tr} e^{\mathcal{A} + \log \mathcal{Y} } 
\leq \mathrm{Tr} e^{\mathcal{A} + \log\left( \mathbb{E} \mathcal{Y} \right) } 
= \mathrm{Tr} e^{\mathcal{A} + \log \left( \mathbb{E} e^{\mathcal{X}} \right) },
\end{eqnarray}
where the inequality is based on Lieb's concavity theorem for T-product tensors obtained by Theorem~\ref{thm:Lieb concavity thm} and Jensen's T-product tensor inequality by Theorem~\ref{thm:Jensen's T-product Inequality intro}.
\end{proof}

\subsection{T-product Tensor Moments and Cumulants}\label{subsec:Tensor Moments and Cumulants}

Since the expectation of a random T-product tensor can be treated as convex combination, expectation will preserve the semidefinite order as
\begin{eqnarray}
\mathcal{X} \succ \mathcal{Y} \mbox{~~almost surely} ~~~ \Rightarrow ~~~
\mathbb{E}(\mathcal{X}) \succ \mathbb{E}(\mathcal{Y}).
\end{eqnarray}
From Jensen's T-product tensor inequality by Theorem~\ref{thm:Jensen's T-product Inequality intro}, we also have 
\begin{eqnarray}
\mathbb{E}(\mathcal{X}^2) \succeq \left(\mathbb{E}(\mathcal{X})\right)^2.
\end{eqnarray}

Suppose a random Hermitian T-product tensor $\mathcal{X}$ having tensor moments of all orders, i.e., $\mathbb{E}(\mathcal{X}^n)$ existing for all $n$, we can define the tensor moment-generating function, denoted as $\mathbb{M}_{\mathcal{X}}(t)$, and the tensor cumulant-generating function, denoted as $\mathbb{K}_{\mathcal{X}}(t)$, for the tensor $\mathcal{X}$ as 
\begin{eqnarray}\label{eq:def mgf and cgf}
\mathbb{M}_{\mathcal{X}}(t) \define \mathbb{E} e^{t \mathcal{X}}, \mbox{~~and~~~}
\mathbb{K}_{\mathcal{X}}(t) \define \log \mathbb{E} e^{t \mathcal{X}},
\end{eqnarray}
where $t \in \mathbb{R}$. Both the tensor moment-generating function and the tensor cumulant-generating function can be expressed as power series expansions:
\begin{eqnarray}\label{eq:mgf and cgf expans}
\mathbb{M}_{\mathcal{X}}(t) = \mathcal{I} + \sum\limits_{n=1}^{\infty} \frac{t^n}{n !}\mathbb{E}(\mathcal{X}^n), \mbox{~~and~~~}
\mathbb{K}_{\mathcal{X}}(t) = \sum\limits_{n=1}^{\infty} \frac{t^n}{n !} \psi_n,
\end{eqnarray}
where $\psi_n$ is called \emph{tensor cumulant}. The tensor cumulant $\psi_n$ can be expressed as a polynomial in terms of tensor moments up to the order $n$, for example, the first cumulant is the mean and the second cumulant is the variance: 
\begin{eqnarray}\label{eq:cumulant mean and var}
\psi_1 = \mathbb{E}(\mathcal{X}), \mbox{~~and~~~}
\psi_2 = \mathbb{E}(\mathcal{X}^2) - (\mathbb{E}(\mathcal{X}))^2.
\end{eqnarray}

Finally, in this work, we also assume that all random variables are sufficiently regular for us to compute their expectations, interchange limits, etc.

\section{Tail Bounds By Concatenation of Lieb's Concavity}\label{sec:Tail Bounds By Concatenation of Liebs Concavity}

The goal of this section is to develop several important tools which will be applied
intensively in the proof of probability inequalities for the extreme eigentule (or eigenvalue) of a sum of independent random T-product tensors. The first tool is the Laplace transform method for T-product tensors discussed in Section~\ref{subsec:Laplace Transform Method for T-product Tensors}, and the second tool is the tail bound for independent sums of random Hermitian T-product tensors presented by Section~\ref{subsec:Tail Bounds for Independent Sums of T-product Tensors}.

\subsection{Laplace Transform Method for T-product Tensors}\label{subsec:Laplace Transform Method for T-product Tensors}

We extend the Laplace transform bound from matrices to T-product tensors based on~\cite{MR1966716}. Following lemma is given to establish the Laplace transform bound for the maximum eigenvalue of a T-product tensor.
\begin{lemma}[Laplace Transform Method for T-product Tensors: Eigenvalue Version]\label{lma: Laplace Transform Method Eigenvalue Version}
Let $\mathcal{X}$ be a random Hermitian T-product tensor. For $\theta \in \mathbb{R}$, we have
\begin{eqnarray}
\mathbb{P} (\lambda_{\max}(\mathcal{X}) \geq \theta) \leq \inf_{t > 0} \Big\{ e^{-\theta t} \mathbb{E}\mathrm{Tr} e^{t \mathcal{X}} \Big\}
\end{eqnarray}
\end{lemma}
\begin{proof}
Given a fix value $t$, we have 
\begin{eqnarray}\label{eq1: lma: Laplace Transform Method eigenvalue}
\mathbb{P} (\lambda_{\max}(\mathcal{X}) \geq \theta) = \mathbb{P} (\lambda_{\max}(t\mathcal{X}) \geq t \theta ) = \mathbb{P} (e^{ \lambda_{\max}(t\mathcal{X})  } \geq e^{ t \theta}  ) \leq e^{-t \theta} \mathbb{E} e^{ \lambda_{\max}(t\mathcal{X}) }. 
\end{eqnarray}
The first equality uses the homogeneity of the maximum eigenvalue map, the second equality comes from the monotonicity of the scalar exponential function, and the last relation is Markov's inequality. Because we have
\begin{eqnarray}\label{eq2: lma: Laplace Transform Method eigenvalue}
e^{\lambda_{\max} (t \mathcal{Y})} = \lambda_{\max}(e^{t \mathcal{Y}}) \leq  \mathrm{Tr} e^{t \mathcal{Y}},
\end{eqnarray}
where the first equality used the spectral mapping theorem from Lemma~\ref{lma:spectral mapping}, and the inequality holds because the exponential of an Hermitian T-product tensor is TPD and the maximum eigenvalue of a TPD tensor is dominated by the trace from Eq.~\eqref{eq:trace def}. From Eqs~\eqref{eq1: lma: Laplace Transform Method eigenvalue} and~\eqref{eq2: lma: Laplace Transform Method eigenvalue}, this lemma is established.
\end{proof}

The Lemma~\ref{lma: Laplace Transform Method Eigenvalue Version} helps us to control the tail probabilities for the maximum eigenvalue of a random Hermitian T-product tensor by utilizing a bound for the trace of the tensor moment-generating function introduced in Section~\ref{subsec:Tensor Moments and Cumulants}.

Since T-product tensors also have notions about eigentuples, we then extend Lemma~\ref{lma: Laplace Transform Method Eigenvalue Version} from eigenvalues version to eigentuples version. We begin with the derivation of Markov's inequality for random vectors.
\begin{lemma}[Markov's inequality for Random Vector]\label{lma:Markov's inequality for Random Vector}
If $\mathbf{X} \in \mathbb{R}^{p}$ is a nonnegative random vector and $\mathbf{a} >  \mathbf{0}$,  then the probability that  $\mathbf{X}$ is at least $\mathbf{a}= [a_i]$ can be bounded as:
\begin{eqnarray}\label{eq:lma:Markov's inequality for Random Vector}
\mathrm{Pr}\left( \mathbf{X} \geq \mathbf{a} \right) \leq \min\limits_{i} \left\{ \frac{ \left( \mathbb{E}\left( \mathbf{X} \right) \right)_i }{a_i}  \right\}
\end{eqnarray}
where $1 \leq i \leq p$.
\end{lemma}
\textbf{Proof:} 
Because we have 
\begin{eqnarray}
\mathbb{E}\left( \mathbf{X} \right) &=& \int\limits_{\mathbf{0}}^{\infty^p} \mathbf{x} f( \mathbf{x}) d  \mathbf{x} =  \int\limits_{\mathbf{0}}^{\mathbf{a}} \mathbf{x} f( \mathbf{x}) d  \mathbf{x}  + \int\limits_{\mathbf{a}}^{\infty^p} \mathbf{x} f( \mathbf{x}) d  \mathbf{x}\nonumber \\
&\geq & \int\limits_{\mathbf{a}}^{\infty^p} \mathbf{x} f( \mathbf{x}) d  \mathbf{x} \geq 
\int\limits_{\mathbf{a}}^{\infty^p} \mathbf{a} f( \mathbf{x}) d  \mathbf{x} = 
\mathbf{a} \int\limits_{\mathbf{a}}^{\infty^p} f( \mathbf{x}) d  \mathbf{x} \nonumber \\
&=& \mathbf{a} \mathrm{Pr}\left( \mathbf{X} \geq \mathbf{a} \right),
\end{eqnarray}
therefore, we have the desired inequality shown by Eq.~\eqref{eq:lma:Markov's inequality for Random Vector}. 
$\hfill \Box$

We are ready to present following lemma about Laplace transform method for T-product tensors with eigentuples. 
\begin{lemma}[Laplace Transform Method for T-product Tensors: Eigentuple Version]\label{lma: Laplace Transform Method Eigentuple Version}
Let $\mathcal{X} \in \mathbb{C}^{m \times m \times p}$ be a random T-positive definite (TPD) tensor and an all one vector $\mathbf{1}_p = [1, 1, \cdots, 1]^\mathrm{T} \in \mathbb{C}^{p}$. Suppose we have 
\begin{eqnarray}\label{eq1:lma: Laplace Transform Method Eigentuple Version}
\frac{1}{p} \lambda_{\max}^{p}(e^{t \mathcal{X}}) + 1 - \frac{1}{p} \leq \mathrm{Tr}(e^{t \mathcal{X}}), 
\end{eqnarray}
where $t > 0$~\footnote{If we scale the random TPD tensor $\mathcal{X}$ as the $\lambda_{\max}(e^{t \mathcal{X}}) = 1$, then Eq.~\eqref{eq1:lma: Laplace Transform Method Eigentuple Version} is always valid.}. Then, for $\mathbf{b} \in \mathbb{R}^p$, we obtain 
\begin{eqnarray}
\mathbb{P} (\mathbf{d}_{\max}(\mathcal{X}) \geq \mathbf{b}) \leq \inf_{t > 0}
 \min\limits_{i} \left\{ \frac{  \mathbb{E} \left( \mathrm{Tr}\left(e^{t \mathcal{X}} \right)\right)   }{ \left(e_{\bigodot}^{ t \mathbf{b}}  \right)_i  }\right\},
\end{eqnarray}
where $\mathbf{d}_{\max}$ is the maximum eigentuple of the TPD tensor $\mathcal{X} $.
\end{lemma}
\textbf{Proof:}
Given a fix value $t$, we have 
\begin{eqnarray}\label{eq1: lma: Laplace Transform Method eigentuple}
\mathbb{P} (\mathbf{d}_{\max}(\mathcal{X}) \geq \mathbf{b}) = \mathbb{P} (\mathbf{d}_{\max}(t\mathcal{X}) \geq t \mathbf{b}) = \mathbb{P} ( e_{\bigodot}^{ \mathbf{d}_{\max}(t\mathcal{X})  } \geq e_{\bigodot}^{ t \mathbf{b}}  ) \leq \min\limits_{i} \left\{ \frac{ \left( \mathbb{E}\left( e_{\bigodot}^{ \mathbf{d}_{\max}(t\mathcal{X})  }   \right) \right)_i   }{ \left( e_{\bigodot}^{ t \mathbf{b}}  \right)_i  }\right\}. 
\end{eqnarray}
The first equality uses the homogeneity of the maximum eigenvalue map, the second equality comes from the monotonicity of the exponential function with operation $\bigodot$ defined in Proposition 2.1 from work~\cite{qi2021tquadratic}, and the last relation is Markov's inequality for random vector obtained from Lemma~\ref{lma:Markov's inequality for Random Vector} since both $\mathbb{E}\left( e_{\bigodot}^{ \mathbf{d}_{\max}(t\mathcal{X})  }   \right)$ and $e_{\bigodot}^{ t \mathbf{b}}$ are vectors with $p$ entries. Then, we have
\begin{eqnarray}\label{eq2: lma: Laplace Transform Method eigentuple}
e_{\bigodot}^{\mathbf{d}_{\max} (t \mathcal{X})} \leq e_{\bigodot}^{\lambda_{\max}(t \mathcal{X}) \mathbf{1}_p} \leq \mathrm{Tr}\left(e^{t \mathcal{X}} \right)   \mathbf{1}_p,
\end{eqnarray}
where the first inequality comes from the relation that $\mathbf{d}_{\max}(t \mathcal{X}) \leq \lambda_{\max}( t \mathcal{X}) \mathbf{1}_p$, and the second inequality holds because $e^{\lambda_{\max}(t\mathcal{X})} = \lambda_{\max}(e^{t \mathcal{X}})$ and the relation $\frac{1}{p} \lambda_{\max}^{p}(e^{t \mathcal{X}}) + 1 - \frac{1}{p} \leq \mathrm{Tr}(e^{t \mathcal{X}})$. From Eqs~\eqref{eq1: lma: Laplace Transform Method eigentuple} and~\eqref{eq2: lma: Laplace Transform Method eigentuple}, this lemma is established.
$\hfill \Box$

\subsection{Tail Bounds for Independent Sums of Random T-product Tensors}\label{subsec:Tail Bounds for Independent Sums of T-product Tensors}

This section will present the tail bound for the sum of independent random T-product tensors and several corollaries according to this tail bound for independent sums. We begin with the subadditivity lemma of tensor cumulant-generating functions.

\begin{lemma}\label{lma:subadditivity of tensor cgfs}
Given a finite sequence of independent random Hermitian T-product tensors $\{ \mathcal{X}_i \}$, where $\mathcal{X}_i \in \mathbb{C}^{m \times m \times p}$, we have 
\begin{eqnarray}\label{eq1:lma:subadditivity of tensor cgfs}
\mathbb{E} \mathrm{Tr} \exp\left( \sum\limits_{i=1}^n t \mathcal{X}_i \right) \leq 
\mathrm{Tr} \exp\left( \sum\limits_i^n \log \mathbb{E} e^{t \mathcal{X}_i} \right),~~\mbox{for $t \in \mathbb{R}$.}  
\end{eqnarray}
\end{lemma}
\textbf{Proof:}
We begin with the following definition for the tensor cumulant-generating function for $\mathcal{X}_i$ as:
\begin{eqnarray}
\mathbb{K}_{i}(t) &\define& \log (\mathbb{E} e^{t \mathcal{X}_i}).
\end{eqnarray}
Then, we define the Hermitian T-product tensor $\mathcal{H}_k$ as 
\begin{eqnarray}\label{eq2:lma:subadditivity of tensor cgfs}
\mathcal{H}_k(t) = \sum\limits_{i = 1}^{k-1} t \mathcal{X}_i + \sum\limits_{i = k+1}^n \mathbb{K}_{i}(t).
\end{eqnarray}

By applying Eq.~\eqref{eq2:lma:subadditivity of tensor cgfs} to Theorem~\ref{thm:Lieb concavity thm} repeatedly for $k = 1,2,\cdots,n$, we have 
\begin{eqnarray}
\mathbb{E} \mathrm{Tr} \exp\left( \sum\limits_{i=1}^n t \mathcal{X}_i \right) 
&=_1& \mathbb{E}_0 \cdots \mathbb{E}_{n-1} \mathrm{Tr} \exp\left(\sum\limits_{i=1}^{n-1} t\mathcal{X}_i + t\mathcal{X}_n \right) \nonumber \\
&\leq&  \mathbb{E}_0 \cdots \mathbb{E}_{n-2} \mathrm{Tr} \exp\left( \sum\limits_{i=1}^{n-1}t \mathcal{X}_i + \log\left( \mathbb{E}_{n - 1}e^{t\mathcal{X}_n} \right) \right) \nonumber \\
& = &  \mathbb{E}_0 \cdots \mathbb{E}_{n-2} \mathrm{Tr} \exp\left( \sum\limits_{i=1}^{n-2} t \mathcal{X}_i + t \mathcal{X}_{n-1} +\mathbb{K}_{n}(t).  \right) \nonumber \\
& \leq &  \mathbb{E}_0 \cdots \mathbb{E}_{n-3} \mathrm{Tr} \exp\left( \sum\limits_{i=1}^{n-2} t \mathcal{X}_i +  \mathbb{K}_{n-1}(t) + \mathbb{K}_{n}(t)  \right) \nonumber \\
\cdots 
& \leq & \mathrm{Tr} \exp\left( \sum\limits_{i = 1}^{n}  \mathbb{K}_{i}(t)  \right)
\end{eqnarray}
where the equality $=_1$ is based on the law of total expectation by defining $\mathbb{E}_i$ as the conditional expectation given $\mathcal{X}_1, \cdots, \mathcal{X}_i$.
$\hfill \Box$

We are ready to present the theorem for the tail bound of independent sums of random Hermitian T-product tensors with respect to the maximum eigenvalue. We recell~\cref{thm:Master Tail Bound for Independent Sum of Random Tensors}

\ThmMasterTailBoundEigenvalue*

%
%%
%\begin{theorem}[Master Tail Bound for Independent Sum of Random T-Tensors for Eigenvalues]\label{thm:Master Tail Bound for Independent Sum of Random Tensors}
%%
%Given a finite sequence of independent Hermitian random tensors $\{ \mathcal{X}_i \}$, we have 
%%
%\begin{eqnarray}\label{eq1:thm:master tail bound}
%%
%\mathrm{Pr} \left( \lambda_{\max} (\sum\limits_{i=1}^n \mathcal{X}_i) \geq \theta \right)
%& \leq & \inf\limits_{t > 0} \Big\{ e^{- t \theta} \mathrm{Tr}\exp \left( \sum\limits_{i=1}^{n} \log \mathbb{E} e^{t \mathcal{X}_i}  \right) \Big\}. 
%\end{eqnarray}
%%
%\end{theorem}
%%
\textbf{Proof:}
By substituting the Lemma~\ref{lma:subadditivity of tensor cgfs} into the Laplace transform bound provided by the Lemma~\ref{lma: Laplace Transform Method Eigenvalue Version}, this theorem is established. 
$\hfill \Box$

Several useful corollaries will be provided based on Theorem~\ref{thm:Master Tail Bound for Independent Sum of Random Tensors}.

\begin{corollary}\label{cor:3_7}
Given a finite sequence of independent Hermitian random tensors $\{ \mathcal{X}_i \} \in \mathbb{C}^{m \times m \times p}$. If there is a function $f: (0, \infty) \rightarrow [0, \infty]$ and a sequence of non-random Hermitian T-product tensors $\{ \mathcal{A}_i \}$ with following condition:
\begin{eqnarray}\label{eq:cond in cor 3_7}
f(t) \mathcal{A}_i \succeq  \log \mathbb{E} e^{t \mathcal{X}_i},~~\mbox{for $t > 0$.}
\end{eqnarray}
Then, for all $\theta \in \mathbb{R}$, we have
\begin{eqnarray}
\mathrm{Pr} \left( \lambda_{\max}\left(\sum\limits_{i=1}^n \mathcal{X}_i \right) \geq \theta \right)
& \leq & mp \inf\limits_{t > 0}\Big\{\exp\left[ - t \theta + f(t)\lambda_{\max}\left(\sum\limits_{i=1}^n \mathcal{A}_i \right) \right] \Big\}
\end{eqnarray}
\end{corollary}
\textbf{Proof:}
From the condition provided by Eq.~\eqref{eq:cond in cor 3_7} and Theorem~\ref{thm:Master Tail Bound for Independent Sum of Random Tensors}, we have 
\begin{eqnarray}
\mathrm{Pr} \left( \lambda_{\max}\left(\sum\limits_{i=1}^n \mathcal{X}_i \right) \geq \theta \right) & \leq & e^{- t \theta} \mathrm{Tr} \exp ( f(t) \sum\limits_{i=1}^n \mathcal{A}_i) \nonumber \\
& \leq & mp e^{- t \theta} \lambda_{\max} \left( \exp (f( t ) \sum\limits_{i=1}^n \mathcal{A}_i) \right) \nonumber \\
& = &  mp  e^{- t \theta} \exp\left( f( t ) \lambda_{\max} \left(\sum\limits_{i=1}^n \mathcal{A}_i \right) \right),
\end{eqnarray}
where the second inequality holds since we bound the trace of a TPD T-product tensor  by the dimension size $m \times p $ multiplied by the maximum eigenvalue; the last equality is based on the spectral mapping theorem since the function $f$ is nonnegative. 
$\hfill \Box$

\begin{corollary}\label{cor:3_9}
Given a finite sequence of independent Hermitian random tensors $\{ \mathcal{X}_i \} \in \mathbb{C}^{m \times m \times p}$. For all $\theta \in \mathbb{R}$, we have
\begin{eqnarray}
\mathrm{Pr} \left( \lambda_{\max}\left(\sum\limits_{i=1}^n \mathcal{X}_i \right) \geq \theta \right)
& \leq &m p  \inf\limits_{t > 0}  \Big\{  \exp\left[ - t \theta +n \log \lambda_{\max} \left( \frac{1}{n} \sum\limits_{i=1}^{n} \mathbb{E} e^{t \mathcal{X}_i} \right)  \right]  \Big\} \nonumber \\
\end{eqnarray}
\end{corollary}
\textbf{Proof:}
From T-tensor logarithm concavity property provided by Lemma~\ref{lma:log monotone and concave}, we have 
\begin{eqnarray}
\sum\limits_{i=1}^{n} \log \mathbb{E} e^{t \mathcal{X}_i} = n \cdot \frac{1}{n}  \sum\limits_{i=1}^{n} \log \mathbb{E} e^{t \mathcal{X}_i} \preceq n \log \left( \frac{1}{n} \sum\limits_{i=1}^{n} \mathbb{E} e^{t \mathcal{X}_i} \right),
\end{eqnarray}
and from the trace exponential monotone property provided by Lemma~\ref{lma:monotonicity of trace func}, we have 
\begin{eqnarray}
\mathrm{Pr} \left( \lambda_{\max}\left(\sum\limits_{i=1}^n \mathcal{X}_i \right) \geq \theta \right)
 \leq  e^{- t \theta} \mathrm{Tr} \exp \left(n \log \left( \frac{1}{n} \sum\limits_{i=1}^n \mathbb{E} e^{ t \mathcal{X}_i }  \right) \right)~~~~~~~~~~~~~~~~~~ \nonumber \\
\leq  m p  \inf\limits_{t > 0}  \Big\{ \exp\left[ - t \theta +n \log \lambda_{\max} \left( \frac{1}{n} \sum\limits_{i=1}^{n} \mathbb{E} e^{t \mathcal{X}_i} \right)  \right] \Big\},
\end{eqnarray}
where the last inequality holds since we bound the trace of a positive-definite tensor by the dimension size $m \times p $ multiplied by the maximum eigenvalue and apply spectral mapping theorem twice.
$\hfill \Box$

Similarly, we can generalize master tail bound for independent sum of random Hermitian T-product tensors for eigenvalue version from Theorem~\ref{thm:Master Tail Bound for Independent Sum of Random Tensors} to master tail bound for independent sum of random Hermitian T-product tensors for eigentuple version by the following Theorem~\ref{thm:master tail bound eigentuple}.

\ThmMasterTailBoundEigentuple*

%%
%\begin{theorem}[Master Tail Bound for Independent Sum of Random T-Tensors for Eigentuple]\label{thm:master tail bound eigentuple}
%%
%Given a finite sequence of independent random Hermitian T-product tensors $\{ \mathcal{X}_i \}$ such that $ \mathcal{X}_i \in \mathbb{C}^{m \times m \times p}$, if $\sum\limits_{i=1}^n t \mathcal{X}_i $ satisfies Eq.~\eqref{eq1:lma: Laplace Transform Method Eigentuple Version},
%we have 
%% 
%\begin{eqnarray}\label{eq1:thm:master tail bound eigentuple}
%%
%\mathrm{Pr} \left( \mathbf{d}_{\max} \left( \sum\limits_{i=1}^n \mathcal{X}_i \right) \geq \mathbf{b} \right)
%& \leq & \inf\limits_{t > 0} \min\limits_{i} \left\{ \frac{ \mathrm{Tr} \exp\left( \sum\limits_{i=1}^n \log \mathbb{E}e^{t \mathcal{X}_i }  \right)     }{ \left( e_{\bigodot}^{t \mathbf{b} }\right)_i } \right\}. 
%\end{eqnarray}
%%
%\end{theorem}
%%
\textbf{Proof:}
By substituting the Lemma~\ref{lma:subadditivity of tensor cgfs} into the Laplace transform bound provided by the Lemma~\ref{lma: Laplace Transform Method Eigentuple Version}, this theorem is established. 
$\hfill \Box$

Some useful corollaries will be provided based on Theorem~\ref{thm:master tail bound eigentuple}.

\begin{corollary}\label{cor:3_7 eigentuple}
Given a finite sequence of independent random Hermitian T-product tensors $\{ \mathcal{X}_i \}$ with dimensions in $\mathbb{C}^{m \times m \times p}$. If there is a function $f: (0, \infty) \rightarrow [0, \infty]$ and a sequence of non-random Hermitian T-product tensors $\{ \mathcal{A}_i \}$ with following condition:
\begin{eqnarray}\label{eq:cond in cor 3_7 eigentuple}
f(t) \mathcal{A}_i \succeq  \log \mathbb{E} e^{t \mathcal{X}_i},~~\mbox{for $t > 0$.}
\end{eqnarray}
Then, for all $\mathbf{b} \in \mathbb{R}^p$ and $\sum\limits_{i=1}^n t \mathcal{X}_i $ satisfing Eq.~\eqref{eq1:lma: Laplace Transform Method Eigentuple Version}, we have
\begin{eqnarray}
\mathrm{Pr} \left( \mathbf{d}_{\max}\left(\sum\limits_{i=1}^n \mathcal{X}_i \right) \geq \mathbf{b} \right)
& \leq & mp \inf\limits_{t > 0} \min\limits_{1 \leq j \leq p} \left\{ \frac{ \exp  \left(  f(t) \lambda_{\max} \left( \sum\limits_{i=1}^n \mathcal{A}_i \right)    \right)    }{ \left( e_{\bigodot}^{t \mathbf{b} }\right)_j } \right\}. 
\end{eqnarray}
\end{corollary}
\textbf{Proof:}
From the condition provided by Eq.~\eqref{eq:cond in cor 3_7 eigentuple} and Theorem~\ref{thm:master tail bound eigentuple}, we have 
\begin{eqnarray}
\mathrm{Pr} \left( \mathbf{d}_{\max}\left(\sum\limits_{i=1}^n \mathcal{X}_i \right) \geq \mathbf{b} \right) & \leq &  \inf\limits_{t > 0} \min\limits_{1 \leq j \leq p} \left\{ \frac{ \mathrm{Tr} \exp\left( f(t) \sum\limits_{i=1}^n \mathcal{A}_i  \right)     }{ \left( e_{\bigodot}^{t \mathbf{b} }\right)_j } \right\} \nonumber \\
 & \leq &  mp \inf\limits_{t > 0} \min\limits_{1 \leq j \leq p} \left\{ \frac{ \lambda_{\max} \left( \exp  \left(  f(t) \sum\limits_{i=1}^n \mathcal{A}_i \right)    \right)    }{ \left( e_{\bigodot}^{t \mathbf{b} }\right)_j } \right\} \nonumber \\
& = &  mp \inf\limits_{t > 0} \min\limits_{1 \leq j \leq p} \left\{ \frac{ \exp  \left(  f(t) \lambda_{\max} \left( \sum\limits_{i=1}^n \mathcal{A}_i \right)    \right)    }{ \left( e_{\bigodot}^{t \mathbf{b} }\right)_j } \right\} 
\end{eqnarray}
where the second inequality holds since we bound the trace of a TPD T-tensor by the eigenvalue size with $m \times p$ multiplied by the maximum eigenvalue; the last equality is based on the spectral mapping theorem since the function $f$ is nonnegative. This corollary is proved.
$\hfill \Box$

\begin{corollary}\label{cor:3_9 eigentuple}
Given a finite sequence of independent random Hermitian T-product tensors $\{ \mathcal{X}_i \}$ with dimensions in $\mathbb{C}^{m \times m \times p}$, a real vector $\mathbf{b} \in \mathbb{R}^p$ and $\sum\limits_{i=1}^n t \mathcal{X}_i $ satisfing Eq.~\eqref{eq1:lma: Laplace Transform Method Eigentuple Version}, we have
\begin{eqnarray}
\mathrm{Pr} \left( \mathbf{d}_{\max}\left(\sum\limits_{i=1}^n \mathcal{X}_i \right) \geq \mathbf{b} \right)
& \leq & m p \inf\limits_{t > 0} \min\limits_{1 \leq j \leq p} \left\{ \frac{\exp \left( n \log \lambda_{\max}\left( \frac{1}{n} \sum\limits_{i=1}^{n}\mathbb{E}e^{t \mathcal{X}_i } \right) \right)     }{ \left( e_{\bigodot}^{t \mathbf{b} }\right)_j } \right\} 
\end{eqnarray}
\end{corollary}
\textbf{Proof:}
From T-tensor logarithm concavity property provided by Lemma~\ref{lma:log monotone and concave}, we have 
\begin{eqnarray}
\sum\limits_{i=1}^{n} \log \mathbb{E} e^{t \mathcal{X}_i} = n \cdot \frac{1}{n}  \sum\limits_{i=1}^{n} \log \mathbb{E} e^{t \mathcal{X}_i} \preceq n \log \left( \frac{1}{n} \sum\limits_{i=1}^{n} \mathbb{E} e^{t \mathcal{X}_i} \right),
\end{eqnarray}
and from the trace exponential monotone property provided by Lemma~\ref{lma:monotonicity of trace func}, we also have 
\begin{eqnarray}
\mathrm{Pr} \left( \mathbf{d}_{\max}\left(\sum\limits_{i=1}^n \mathcal{X}_i \right) \geq  \mathbf{b} \right) 
& \leq &  \inf\limits_{t > 0} \min\limits_{1 \leq j \leq p  } \left\{ \frac{ \mathrm{Tr} \exp\left( \sum\limits_{i=1}^n \log \mathbb{E}e^{t \mathcal{X}_i }  \right)     }{ \left( e_{\bigodot}^{t \mathbf{b} }\right)_j } \right\} \nonumber \\
& \leq & \inf\limits_{t > 0} \min\limits_{1 \leq j \leq p} \left\{ \frac{ \mathrm{Tr} \exp\left( n \log \left( \frac{1}{n} \sum\limits_{i=1}^{n} \mathbb{E} e^{t \mathcal{X}_i} \right)   \right)     }{ \left( e_{\bigodot}^{t \mathbf{b} }\right)_j } \right\} \nonumber \\
& \leq & m p \inf\limits_{t > 0} \min\limits_{1 \leq j \leq p} \left\{ \frac{\exp \left( n \log \lambda_{\max}\left( \frac{1}{n} \sum\limits_{i=1}^{n}\mathbb{E}e^{t \mathcal{X}_i } \right) \right)     }{ \left( e_{\bigodot}^{t \mathbf{b} }\right)_j } \right\} 
\end{eqnarray}
where the last inequality holds since we bound the trace of a TPD tensor by the eigenvalue size with $m \times p$ multiplied by the maximum eigenvalue; and spectral mapping theorem for $\log$ and $\exp$ functions. 
$\hfill \Box$

\section{Courant-Fischer Theorem under T-product Tensors and Minimum Eigenvalue/Eigentuple}\label{sec:Courant-Fischer Theorem under T-product Tensors and Minimum}

In this section, Courant-Fischer theorem for T-product tensors will be proved and this theorem will be used to show the relationship between the maximum eigentuple and the minimum eigentule of TPD T-product tensors. Let us recall~\cref{thm:Courant-Fischer Theorem under T-product}. 

\ThmCourantFischerTProduct*

%\begin{theorem}[Courant-Fischer Theorem under T-product]\label{thm:Courant-Fischer Theorem under T-product}
%%
%Let $\mathcal{A} \in \mathbb{C}^{m \times m \times p}$ be a Hermitian T-product tensor with eigentuples $\mathbf{d}_1 \geq \mathbf{d}_2 \geq \cdots \geq \mathbf{d}_n$. Let $\{\mathbf{U}_{j}^{[k]} \} \in \mathbb{C}^{m \times p }$ be orthnomal matrices for $1 \leq j \leq m$ and $1 \leq k \leq p$, $S_k$ be the space spanned by $\{\mathbf{U}_{j}^{[k]} \}$  for $1 \leq j \leq k$ and $1 \leq k \leq p$, and $T_k$ be the space spanned by $\{\mathbf{U}_{j}^{[k]} \}$  for $k \leq j \leq m$ and $1 \leq k \leq p$. Then, we have 
%%
%\begin{eqnarray}
%\mathbf{d}_k &=& \max\limits_{\substack{S_k \subseteq \mathbb{C}^{m \times p} \\ \dim(S_k) = k \times p } } \min\limits_{\mathbf{X} \in S_k} \left(\mathbf{X}^{\mathrm{H}} \star \mathcal{A} \star \mathbf{X} \right) \bigg /_{\bigodot} \left( \mathbf{X}^{\mathrm{H}} \star \mathbf{X} \right) \nonumber \\
%&=& \min\limits_{\substack{ T_k \subseteq \mathbb{C}^{m \times p} \\ \dim(T_k) = (m - k +1) \times p } } \max\limits_{\mathbf{X} \in T_k} \left(\mathbf{X}^{\mathrm{H}} \star \mathcal{A} \star \mathbf{X} \right) \bigg /_{\bigodot} \left( \mathbf{X}^{\mathrm{H}} \star \mathbf{X} \right),
%\end{eqnarray}
%%
%where  $\bigg /_{\bigodot}$ is the division (inverse operation) under $\bigodot$ defined in Proposition 2.1 from work~\cite{qi2021tquadratic}. 
%%
%\end{theorem}
%%
\textbf{Proof:}
We will just prove the first characterization of $\mathbf{d}_k$. The other can be proved similarly. 

First, we wish to show that $\mathbf{d}_k$ is achievable. As $S_k$ is the space spanned by
$\{\mathbf{U}_{j}^{[k]} \}$  for $1 \leq j \leq k$ and $0 \leq l \leq p-1$. For every $\mathbf{X} \in S_k$, we can express $\mathbf{X}$ as 
\begin{eqnarray}
\mathbf{X} &=& \sum\limits_{j=1}^{k}\sum\limits_{l=0}^{p-1} \alpha_j^{[l]} \mathbf{U}_{j}^{[l]}.
\end{eqnarray}
Then, we have
\begin{eqnarray}
\left(\mathbf{X}^{\mathrm{H}} \star \mathcal{A} \star \mathbf{X} \right) \bigg /_{\bigodot} \left( \mathbf{X}^{\mathrm{H}} \star \mathbf{X} \right) &=& 
 \sum\limits_{j=1}^{k}\sum\limits_{l=0}^{p-1} \left( \alpha_j^{[l]}\right)^2  \mathbf{d}_{j} \bigg /_{\bigodot}  \sum\limits_{j=1}^{k}\sum\limits_{l=0}^{p-1}   \left( \alpha_j^{[l]}\right)^2    \mathbf{e} \nonumber \\
&\geq &   \sum\limits_{j=1}^{k}\sum\limits_{l=0}^{p-1} \left( \alpha_j^{[l]}\right)^2  \mathbf{d}_{k}  \bigg /_{\bigodot}  \sum\limits_{j=1}^{k}\sum\limits_{l=0}^{p-1}   \left( \alpha_j^{[l]}\right)^2    \mathbf{e} \nonumber \\
& = &  \mathbf{d}_k
\end{eqnarray}
where $\mathbf{e} = (1, 0, \cdots, 0)^{\mathrm{T}} \in \mathbb{C}^{p}$. 

To verify that this is the maximum eigentuple, as $T_k$ is the space spanned by $\{\mathbf{U}_{j}^{[l]} \}$  for $k \leq j \leq m$ and $0 \leq l \leq p-1$, for any $S_k$ with dimension $k \times p$ the intersection of $S_k$ with $T_k$ is non-empty. Then, we also have 
\begin{eqnarray}
\min\limits_{\mathbf{X} \in S_k}  \left(\mathbf{X}^{\mathrm{H}} \star \mathcal{A} \star \mathbf{X} \right) \bigg /_{\bigodot} \left( \mathbf{X}^{\mathrm{H}} \star \mathbf{X} \right)
& \leq &  \min\limits_{\mathbf{X} \in S_k \cap T_k}  \left(\mathbf{X}^{\mathrm{H}} \star \mathcal{A} \star \mathbf{X} \right) \bigg /_{\bigodot} \left( \mathbf{X}^{\mathrm{H}} \star \mathbf{X} \right).
\end{eqnarray}
Any such $\mathbf{X}$ can be expressed as 
\begin{eqnarray}
\mathbf{X} &=& \sum\limits_{j=k}^{m}\sum\limits_{l=0}^{p-1} \alpha_j^{[l]} \mathbf{U}_{j}^{[l]},
\end{eqnarray}
then, we have 
\begin{eqnarray}
\left(\mathbf{X}^{\mathrm{H}} \star \mathcal{A} \star \mathbf{X} \right) \bigg /_{\bigodot} \left( \mathbf{X}^{\mathrm{H}} \star \mathbf{X} \right) &=& 
 \sum\limits_{j=k}^{m}\sum\limits_{l=0}^{p-1} \left( \alpha_j^{[l]}\right)^2  \mathbf{d}_{j} \bigg /_{\bigodot}  \sum\limits_{j=k}^{m}\sum\limits_{l=0}^{p-1}   \left( \alpha_j^{[l]}\right)^2    \mathbf{e} \nonumber \\
&\leq &   \sum\limits_{j=k}^{m}\sum\limits_{l=0}^{p-1} \left( \alpha_j^{[l]}\right)^2  \mathbf{d}_{k}  \bigg /_{\bigodot}  \sum\limits_{j=k}^{m}\sum\limits_{l=0}^{p-1}   \left( \alpha_j^{[l]}\right)^2    \mathbf{e} \nonumber \\
& = &  \mathbf{d}_k.
\end{eqnarray}
Therefore, all subspace of $S_k$ with dimension $k \times p$, we have 
\begin{eqnarray}
\min\limits_{\mathbf{X} \in S_k} \left(\mathbf{X}^{\mathrm{H}} \star \mathcal{A} \star \mathbf{X} \right) \bigg /_{\bigodot} \left( \mathbf{X}^{\mathrm{H}} \star \mathbf{X} \right)  \le \mathbf{d}_k.
\end{eqnarray}
This theorem is proved since $\mathbf{d}_k$ is achievable and is the maximum eigentuple. 
$\hfill \Box$

By applying Theorem~\ref{thm:Courant-Fischer Theorem under T-product}, we have 
following relations: 
\begin{eqnarray}\label{eq:min eigen and max eigen relations}
\mathbf{d}_{\min}(\mathcal{X}) = - \mathbf{d}_{\max}( - \mathcal{X})~~\mbox{and}~~ 
\lambda_{\min}(\mathcal{X}) = - \lambda_{\max}( - \mathcal{X})
\end{eqnarray}

\section{Conclusion}\label{sec:Conclusion}
In this Part I work, we try to establish following inequalities about T-product tensors: (1) trace function nondecreasing/convexity; (2) Golden-Thompson inequality for T-product tensors; (3) Jensen’s T-product inequality; (4) Klein’s T-product inequality. All these inequalities are used to generalize celebrated Lieb’s concavity theorem from matrices to T-product tensors. Then, this new version of Lieb's concavity theorem under T-product tensor is utilized to build master tail bounds for the maximum eigenvalue and the maximum eigentuple induced by independent sums of random Hermitian T-product. In order to find the relationship between the minimum eigentuple and the maximum eigentuple, we also extended the Courant-Fischer Theorem from matrices to T-product tensors. How these new inequalities and Courant-Fischer Theorem under T-product are used to derive new tail bounds of the extreme eigenvalue and eigentuple for sums of random T-product tensors is the main goal of our Part II paper.

%\newpage
%\pagenumbering{arabic}
%\setcounter{page}{1}
%\renewcommand{\thepage}{F-\arabic{page}}

\bibliographystyle{IEEETran}
\bibliography{Random_TBounds_TProduct_Bib}

\end{document}